\documentclass[a4paper,12pt]{amsart}

\usepackage{amssymb,amscd,amsthm,mathrsfs}
\usepackage[ansinew]{inputenc}
\usepackage[centertags]{amsmath}
\usepackage[all]{xypic}
\usepackage{graphicx}
\usepackage{color}
\usepackage{transparent}
\usepackage{xifthen}
\usepackage{pdfpages}
\usepackage{import}

   \def\DD{{\mathbb D}}
   \def\HH{{\mathbb H}}

 \def\RR{{\mathbb R}}  \def\TT{{\mathbb T}}
   
 \def\ZZ{{\mathbb Z}}

  \def\cG{\mathcal{G}}  
\def\cB{\mathcal{B}}  \def\cH{\mathcal{H}}  \def\cT{\mathcal{T}}
\def\cC{\mathcal{C}}  \def\cI{\mathcal{I}} \def\cO{\mathcal{O}} 
\def\cD{\mathcal{D}}    
    
\def\cF{\mathcal{F}}  \def\cL{\mathcal{L}}

\newtheorem*{teo*}{Theorem}

\newtheorem{teo}{Theorem}[section]

\newtheorem{quest}{Question}

\newtheorem{lema}[teo]{Lemma}
\newtheorem{prop}[teo]{Proposition}

\newcommand{\bi}{\begin{itemize}}
\newcommand{\ei}{\end{itemize}}

\theoremstyle{definition}
\newtheorem{defi}[teo]{Definition}
\theoremstyle{remark}
\newtheorem{obs}[teo]{Remark}
\newtheorem{ej}[]{Example}

\newcommand{\eps}{\varepsilon}

\newcommand{\wfws}{\widetilde{\cF^{ws}}}
\newcommand{\wfwu}{\widetilde{\cF^{wu}}}
\newcommand{\mt}{\widetilde{M}}

\newcommand{\comment}[1]{}

\author[R. Potrie]{Rafael Potrie}
\address{CMAT, Facultad de Ciencias, Universidad de la Rep\'ublica, Uruguay and CNRS, IRL-IFUMI}
\urladdr{www.cmat.edu.uy/$\sim$rpotrie}
\email{rpotrie@cmat.edu.uy}

\title[Anosov flows in dimension 3]{Anosov flows in dimension 3: an outside look}
\thanks{ R.P. was partially supported by CSIC, ANII.}

\begin{document}

\begin{abstract}
These notes were intended as support material for a minicourse on Anosov flows in the conference "Symplectic geometry and Anosov flows'' which took place in Heidelberg in July 2024 organized by Peter Albers, Jonathan Bowden and Agust\'in Moreno. I took the invitation to present the subject as asking from an outsider view of the subject, given the fact that my research uses both ideas and results from the theory of Anosov flows. The point of view of the course is to provide an overview of the main results and questions in the subject, with emphasis on the interaction with topology, geometry, specially symplectic geometry and contact aspects of the theory. Some detail is given in the presentation of the Barbot-Fenley theory of leaf spaces. Hopefully the notes will contribute in gaining a working knowledge of the theory and its many beautiful connections. 


\bigskip


\medskip

\end{abstract}

\maketitle

\section{Introduction}

Anosov flows in 3-manifolds have a long and rich history. Motivated by the successful study of geodesic flows in negative curvature, Anosov and Sinai proposed a general definition that, by that time, included also some suspension flows. An appendix to that work, by Margulis \cite{Margulis} provided the first topological obstructions to the existence of Anosov flows. Haefliger and Novikov's work on foliations allowed also to produce some obstructions and this was exploited by Plante and Thurston \cite{PlanteThurston}. We refer to \cite{PotrieNot} for more on this, as we shall concentrate on what came next. 

Several new examples started to appear \cite{FW, HandelThurston, Goodman} which allowed for predicting an even deeper connection with topology and foliations. In \cite{Ghys}  the first classification result was proved, showing that in circle bundles, all Anosov flows are orbit equivalent to finite covers of geodesic flows in negative curvature. These notes will concentrate on the theory which started developing in the 90's, with the two influential papers \cite{Barbot-leaf,Fenley} which started to develop a systematic study of the foliations associated to Anosov flows. In a sense, this theory has been somewhat distanced from other work on foliations, and we hope to concentrate here on some of these relations, both having Anosov foliations as important examples of foliations, but also to use the theory of foliations to understand Anosov foliations. Recently, some connections have been made with contact structures, since pairs of transverse foliations give rise to bi-contact structures, and we hope that the introduction here can be useful for people interested in these connections. Naturally, these notes contain more material than can be reasonably covered in 3 lectures, but I hope this can be useful to complement the lectures. 

The plan of the notes is as follows: 

\begin{itemize}
\item In \S~\ref{s.definitions} we give precise definitions and state the basic results that will be the starting point of our study. 

\item In \S~\ref{s.topoobs} we briefly explain some classical topological obstructions to the existence of Anosov flows in some 3-manifolds. 

\item In \S~\ref{s.barbotfenley} we review the Barbot-Fenley theory of leaf spaces and obtain some of the main basic results on the interaction between leaf spaces, bi-foliated plane actions, etc. 

\item In  \S~\ref{s.geomleaves}  we study the geometric properties of leaves of the foliations and how the orbits sit inside these leaves. 

\item In \S \ref{s.universal} we study the structure at infinity, producing some compactifications and circles at infinity. 

\item \S \ref{s.examples}  is devoted to a (very brief) presentation of examples and constructions, trying to relate to some of the aspects that have been discussed. 

\item In \S \ref{s.contact} we discuss the contact property and its relation with the leaf spaces. 

\item Finally, in \S \ref{s.otherthings} we discuss some other directions that we have neglected, being aware that there are still many others. 
\end{itemize}

Throughout we pose many questions, many of which may be quite naive, but that we expect can provide interesting lines for further research. 

\medskip

{\small \emph{Acknowledgements:}  I take the opportunity to thank the organizers of the event for the opportunity to discuss this subject. I'd like to thank Christian Bonatti, Elena Gomes and Th\'eo Marty for letting me use some of their figures. Discussions over the years with  Thierry Barbot, Thomas Barthelm\'e, Christian Bonatti, Sergio Fenley, Elena Gomes, Katie Mann, Santiago Martinchich and Mario Shannon were important to understand some points and to shape my point of view on the subject, which is still work in progress. During the conference I had the luck to discuss and learn some aspects and points of view from Jonathan Bowden, Kai Cieliebak, Pierre Dehornoy, Anna Florio, Surena Hozoori, Th\'eo Marty, Thomas Massoni, Agust\'in Moreno, Federico Salmoraghi  besides the people mentioned before and this had a positive impact on this notes. Finally, I thank the referee for a careful read and many relevant suggestions.}


\section{Definitions and basic properties}\label{s.definitions}
We will provide definitions in any dimensions, but quickly specialize to dimension 3 where most of the theory is developed. We will state many classical results of the theory without giving specific references, we refer the reader to \cite{Barthelme} for a more careful presentation. 

\subsection{Classical definition} 
Let $M$ be a closed smooth manifold. We say that (smooth)\footnote{We will not intend to work in the most general setting, one can indeed define this for less regular vector fields.} a vector field $X$ in $M$ is an \emph{Anosov vector field} if denoting $X_t$ the flow generated by $X$ it holds that there is a $DX_t$-invariant (continuous) splitting $TM = E^s \oplus \RR X \oplus E^u$ satisfying that there is $T>0$ so that for vectors $v^s \in E^s$ and $v^u \in E^u$ one has that: 

\begin{equation}\label{eq:AnosovSmooth}
\|DX_T v^s \| \leq \frac{1}{2} \|v^s\| \text{ and } \|DX_T v^u \| \geq 2 \|v^u\|. 
\end{equation}

\begin{figure}[ht]
\begin{center}
\includegraphics[scale=0.70]{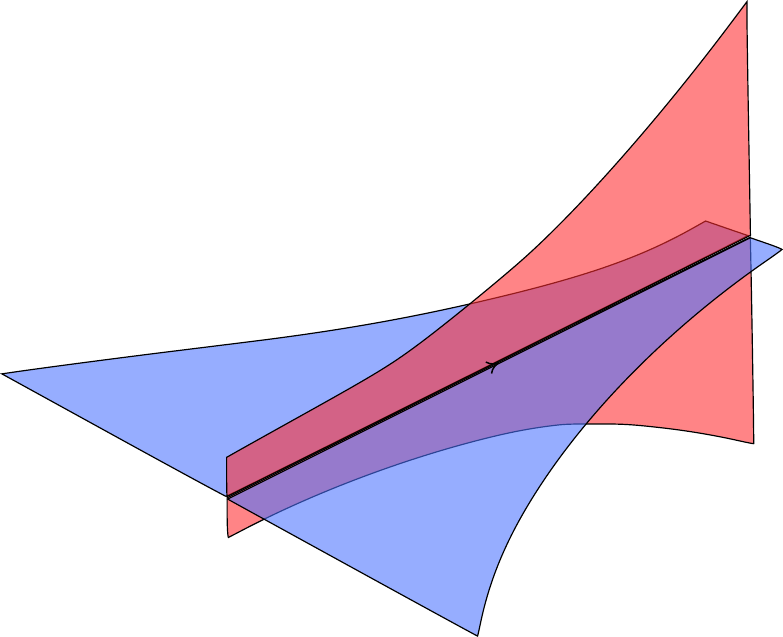}
\begin{picture}(0,0)
\end{picture}
\end{center}
\vspace{-0.5cm}
\caption{{\small Local picture of an Anosov flow. (Figure by Elena Gomes)}\label{f.flujo}}
\end{figure}

Two important examples, which were those which motivated the definition are the following (we leave checking the fact that these are Anosov as a \emph{challenging} exercise): 

\begin{ej}[Suspensions]\label{ej-susp}
Consider a linear map $A \in \mathrm{SL}_d(\ZZ)$ so that all its eigenvalues have modulus different from one. It induces a diffeomorphism $f_A : \TT^d \to \TT^d$ where we identify $\TT^d \cong \RR^d/_{\ZZ^d}$. Consider the manifold $M_A = \TT^d \times \RR /_{\sim}$ where $(x,s) \sim (f_A^k(x), s- k)$ for every $x \in \TT^d, t \in \RR, k \in \ZZ$. The flow $\varphi_t: M_A \to M_A$ given by $\varphi_t([(x,s)]) = [(x,s+t)]$ is Anosov. We call such examples \emph{suspensions}. 
\end{ej}

\begin{ej}[Geodesic flows]\label{ej-geod}
If $M$ is a closed Riemannian manifold with negative sectional curvature everywhere, it was shown by Anosov \cite{Anosov} that for such manifolds, the geodesic flow is Anosov. The proof that such a flow is Anosov simplifies for surfaces (which give rise to Anosov flows in 3-manifolds) and can be further simplified by assuming that the metric has constant negative curvature (when one can identify the unit tangent bundle with a quotient of $\mathrm{PSL}_2(\RR)$). 
\end{ej}

There are many equivalent definitions. One is to say that the time one map of the flow is \emph{partially hyperbolic} which also has several equivalent definitions. We refer to \cite{CP} for several properties and equivalences. 

Here are some useful exercises to test the understanding of the definition: 

\begin{itemize}
\item If $X$ is Anosov, then so is $-X$ with the dimension of the bundles exchanged. 
\item If $\rho : M \to \RR_{>0}$ is a smooth function and $Y= \rho X$ with $X$ being an Anosov vector field, then $Y$ is also an Anosov vector field. The bundles verify $E^s_X \oplus \RR X = E^s_Y \oplus \RR Y$ and $\RR X \oplus E^u_X = \RR Y \oplus E^u_Y$ but it may be that $E^s_X \neq E^s_Y$ and $E^u_X \neq E^u_X$. Try to give a formula for $E^\sigma_Y$ in terms of $E^\sigma_X$ and $\rho$. 
\end{itemize}

There are several reasons that justify using a definition that depends on the differentiable structure and relying on properties that one can detect in the tangent bundle. The first one\footnote{A related reason that motivates studying Anosov flows in dimension 3 is that every $C^1$-robustly transitive flow in a 3-manifold must be an Anosov flow \cite{Doering}.} is that under this definition, the vector fields verifying this property are an open set and there is a finite procedure that certifies that a vector field belongs to this class (we refer the reader to the notion of \emph{cone-fields} expanded in \cite{CP}). In fact, the class is more than open: 

\begin{teo}[Structural stability of Anosov flows]\label{sstability}
The set of Anosov vector fields is open in the $C^1$-topology and any pair of $C^1$-close Anosov vector fields are orbit equivalent. 
\end{teo}

Being orbit equivalent means that there is a homeomorphism of $M$ sending (oriented) orbits of one flow to (oriented) orbits of the other. This also justifies the fact that we chose to work with smooth vector fields: if we want to understand the topological and differential geometric properties of a less regular Anosov flow, we can pick a flow which is $C^1$-close and smooth, and this flow will be orbit equivalent to the original one. Note however that, as mentioned above, the strong bundles $E^s$ and $E^u$ do depend on the parametrization, so, some finer properties (such as mixing, or exponential mixing) are more subtle and rely on the specific choice of parametrizations. We will not delve into these issues. 

The reason the definition is useful is because the geometric structure preserved by the derivative of the flow integrates to actual information in the manifold. For simplicity, we will state the result assuming that the invariant bundles are all orientable, this can be always achieved by taking a finite (four-fold\footnote{A two-fold cover gives that $M$ is orientable, and then, a further two-fold cover orients $E^s$, which is enough since $\RR X$ has a cannonical orientation, so if the weak-stable bundle is orientable and the manifold too, then, so is $E^u$.}) cover. We will assume familiarity with the concept of foliations (see figure \ref{f.fol}, and we refer to \cite{Calegari, CamachoLins,CandelConlon}). 

\begin{teo}[Stable manifold theorem]\label{teo.stableman}
Let $X$ be an Anosov vector field and $X_t$ its associated flow with orientable bundles. Then, the bundles $E^s \oplus \RR X$ and $E^s$ are uniquely integrable into foliations $\cF^{ws}, \cF^{ss}$. Moreover, these foliations verify:
\begin{enumerate}
\item they are $X_t$-invariant, that is $X_t (\cF^\sigma (x)) = \cF^\sigma (X_t(x))$ for $\sigma=ss,ws$),
\item $\cF^{ws}(x) = \bigcup_{t \in \RR} X_t(\cF^{ss}(x)) = \bigcup_{t \in \RR} \cF^{ss}(X_t(x))$, 
\item every leaf of $\cF^{ss}$ is an injectively immersed copy of $E^s$,
\item every leaf of $\cF^{ws}$ is either an injectively immersed copy of $E^s \oplus \RR X$ or an injectively immersed copy of $E^s \times S^1$,
\item a leaf of $\cF^{ws}$ is homeomorphic to  $E^s \times S^1$ only when the leaf contains a periodic orbit of $X_t$ (and this is unique in the leaf). Moreover, the holonomy of the foliation along the periodic orbit (in the direction of the flow) is expanding\footnote{It may be confusing that the holonomy of the weak stable foliation is expanding rather than contracting, but the point it that when going in the flow direction, the transverse direction to the foliation is the unstable, and thus it expands.}. 
\item If $y \in \cF^{ss}(x)$ then it holds that $\lim_{t \to \infty} \frac{1}{t} \log( d(X_t(y),X_t(x)) )< 0$.
\end{enumerate} 
\end{teo}

We call $\cF^{ws}$ the weak stable foliation and $\cF^{ss}$ the strong stable foliation. In  general, we will not use $\cF^{ss}$. Of course the same result holds for constructing the weak unstable foliation $\cF^{wu}$ and the strong unstable one $\cF^{uu}$. There are very subtle issues about the regularity of these foliations. While we shall ignore mostly these issues here, we want to point out that these issues are not mere technical, as the lack of regularity of these foliations is structural about these systems (even if the systems are as smooth as possible). Indeed, it is known that if an Anosov flow has sufficiently smooth weak stable and weak unstable foliations, then, it has to be algebraic in some quite strong sense \cite{Ghys2}, so, most examples do not have smooth foliations (the results in \cite{Ghys2} are in dimension 3, and while there are important results in higher dimensions, the equivalent result is not known in higher dimensions). 

\begin{figure}[ht]
\begin{center}
\includegraphics[scale=0.85]{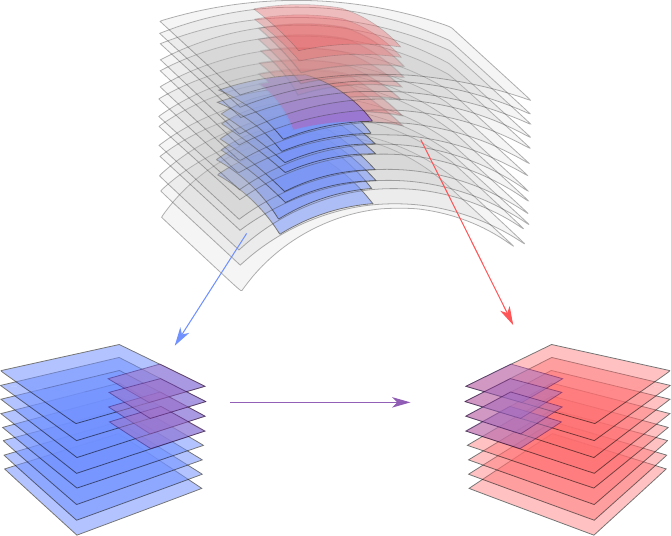}
\begin{picture}(0,0)
\end{picture}
\end{center}
\vspace{-0.5cm}
\caption{{\small Foliation charts. (Figure by Elena Gomes)}\label{f.fol}}
\end{figure}

Proofs of Theorems \ref{sstability} and \ref{teo.stableman} can be found for instance in \cite{FH,CP}. 

\subsection{Topological definition} 

From the point of view of understanding the interaction between the topology of the manifold and the kind of flows it admits, the definition of Anosov vector fields or flows may seem artificial. Also, there are many constructions which are easier to perform in lower regularity and the fact that the orbits converge exponentially fast is usually not very relevant to the topological study (because it concerns an infinitesimal property). This is why in some cases, it is convenient to use a topological definition. It has been an open question for a while if these definitions coincide, and this is still open in general, but a recent breakthrough showed that in most cases of interest in dimension 3, both definitions coincide \cite{Shannon}. 

Let us give briefly the definition in any dimension. We refer the reader to \cite[\S 3.6]{Martinchich} for a more detailed discussion. 

\begin{defi}
A flow $\phi_t: M \to M$ with $C^1$-orbits of a closed manifold is a \emph{topological Anosov flow} if it preserves two (topologically) transverse foliation $\cF^{ws}$ and $\cF^{wu}$ intersecting in the flowlines and with the property that flowlines in $\cF^{ws}$ are forward asymptotic and flowlines in $\cF^{wu}$ are backward asymptotic. 
\end{defi}

There are many subtleties in this definition, which vanish when one restricts to dimension 3. In that case, it is shown in \cite{Shannon} that when the flow is transitive then it must be orbit equivalent to a flow generated by an Anosov vector field as in the previous section. 

It is also shown\footnote{The corresponding statement for Anosov flows, is given in \cite{Fenley} with a very different proof that does not extend to the topological Anosov setting.} in \cite[\S 5]{BFP} that for topological Anosov flows in 3-manifolds leaves of $\cF^{ws}$ and $\cF^{wu}$ with the metric induced by the ambient metric are uniformly Gromov hyperbolic and orbits are quasi-geodesic in each leaf. It is also shown that one can choose some smooth structures so that leaves are smooth. We refer the reader to that paper for more details. 

In general, one can regard a topological Anosov flow as one which verifies (at least most of) the properties of Theorem \ref{teo.stableman}. We note that in dimension 3 we will not really use the strong foliations, but it is sometimes useful to know that there is always a 1-dimensional foliation transverse to a (topological) foliation by surfaces (see \cite{HectorHirsch}).  

One useful thing about working with topological Anosov flows is that when looking at equivalence up to orbit equivalence, it is relevant to know that orbit equivalence of Anosov flows needs not in general be even $C^1$ (as thus important results such as Theorem \ref{sstability} would become false). The lack of regularity of orbit equivalences makes some natural questions to be still badly understood: 

\begin{quest}
Let $\phi_t^1, \phi_t^2: M \to M$ be (smooth) Anosov flows\footnote{For an answer in the case of Anosov diffeomorphisms, see \cite{FG}.} which are orbit equivalent by an orbit equivalence homotopic to identity. Is it true that there is a path of Anosov flows connecting $\phi_t^1$ to $\phi_t^2$?
\end{quest}

In what follows, we will restrict to 3-dimensional manifolds and state results that hold for topological Anosov flows, but whenever it simplifies the exposition, we will assume that the flow is actually Anosov (comes from an Anosov vector field) without further explanation. We will from now on denote Anosov (and topologically Anosov) flows by $\phi_t: M \to M$. 


\section{Some topological obstructions}\label{s.topoobs} 

Very quickly after the notion of Anosov flow was devised, Margulis \cite{Margulis} proved a beautiful result which provided restrictions on the topology of a manifold that carries an Anosov flow. The proof was elementary, using only the basic properties from  Theorem \ref{teo.stableman} and later a new proof appeared in \cite{PlanteThurston} that used deeper properties from codimension one foliations and extended Margulis result to higher dimensions in some situations (see also \cite{Paternain} for extensions allowing to weaken the Anosov property). Essentially, it states that to be able to have exponential expansion locally everywhere, one needs enough global space. 

\begin{teo}[Margulis]\label{t.margulis}
Let $M$ be a closed 3-manifold admitting a (topological) Anosov flow $\phi_t$ then the fundamental group of $M$ has exponential growth. 
\end{teo}

This result has been explained with a modern point of view in \cite{PotrieNot}. It excludes some families of 3-manifolds as possible manifolds admitting Anosov flows. For instance, the 3-dimensional sphere, lens spaces, the three dimensional torus, $S^2 \times S^1$ or nilmanifolds cannot admit Anosov flows. 

The argument in \cite{PlanteThurston} (which incorporates the use of the classical Haefliger argument, or Novikov's theorem, see \cite{CandelConlon}) also allows to deduce that a manifold admitting an Anosov flow must be irreducible and moreover, its universal cover be homeomorphic to $\RR^3$ (this is due to a theorem of Palmeira, see also \cite{CandelConlon, Calegari}). This excludes other classes of manifolds, such as connected sums. 

When this result appeared, very few examples of Anosov flows were known. In the end of the 70's and 80's new examples started to appear, in particular, examples on many new manifolds obtained by surgery constructions. The most mysterious ones remain those in hyperbolic 3-manifolds. 

One could expect that the fact that there is a three way splitting of the tangent bundle could impose some obstructions, however, every (orientable) 3-manifold has a trivial tangent bundle. Nevertheless, in some situations, using some properties of the foliations, one can find obstructions not covered by the previous arguments in some manifolds. An example of such an argument, is the one determining which circle bundles over surfaces admit Anosov flows (this has been extended by Barbot to Seifert spaces, see \cite{Barbot}): 

\begin{teo}[Ghys \cite{Ghys}]\label{teo.ghys}
Let $M$ be a circle bundle over a surface $S$ of genus $g$. If $M$ admits an Anosov flow, then, it is a finite cover of the unit tangent bundle of $S$. 
\end{teo}

While the result in \cite{Ghys} is stronger, and provides a strong classification of Anosov flows in circle bundles (the completion of this classification is contained in the recent \cite{BarbotFenley3}), this result can be seen as a Euler class obstruction result. The proof of the statement above can be summarized as follows (see also \cite{HPS}), first one shows that the weak stable foliation must be \emph{horizontal} (i.e. transverse to fibers of the circle bundle after isotopy) and then one can use the fact that the flow vector field can then produce a map into the unit tangent bundle of $S$ covering the identity. The existence of such a map implies that $M$ has to be a finite cover of $T^1S$ and gives the result (see \cite[\S 6]{HPS} for a detailed proof from this point of view). 

However, the deepest connections between the topology and geometry of manifolds admitting Anosov flows has been developed since the 1990's thanks to a more detailed study of the interaction between the weak stable and weak unstable foliations, their geometry and topology which was initiated in the very influential papers \cite{Barbot-leaf,Fenley} which we wish to describe in some detail in this notes.


\section{Foliations and leaf spaces}\label{s.barbotfenley}
The study of leaf spaces of the weak stable and weak unstable foliations of Anosov flows can be traced back to \cite{Verjovsky} who identified the interest in showing that such a leaf space can be Hausdorff in some situations and its consequences. Its systematic study started certainly in the influential and very rich works of Barbot and Fenley \cite{Barbot-leaf,Fenley}. In this section, we will focus on describing a selection of these results and venture to pose questions that highlight that, despite the remarkable progress made in the field, there remain fundamental questions to understand. All the results here without reference are due to Barbot, Fenley, or both. 

A key aspect of Barbot's and Fenley's work is to expose the fact that while the possible types of foliations that Anosov flows can have is quite rich, it also has several restrictions that makes their study simultaneously rigid in order to get strong interactions between the topology of the manifold and Anosov flows in such a manifold, and flexible enough to allow for very diverse and rich behavior. Not every foliation on a 3-manifold is the weak-stable foliation of an Anosov flow, but many are, and this makes its study very rich (we suggest looking at the diagram in the first page of \cite{HandelThurston}).

\subsection{Individual leaf spaces} 

Let $\phi_t : M \to M$ be an Anosov flow (or topological Anosov flow) on a closed 3-manifold $M$. Theorem \ref{teo.stableman} ensures the existence of transverse foliations $\cF^{ws}$ and $\cF^{wu}$ of codimension one, which (under some orientability assumptions) is by cylinders and planes. The classical Novikov theorem \cite{CandelConlon,Calegari} implies that in the universal cover, one has a foliation by planes and that no transversal intersects the same leaf twice. 

This allows the definition of a leaf space of each of the foliations. We state this as: 

\begin{lema}\label{lem.leafspace}  
Let $\wfws$ be the lift of $\cF^{ws}$ to the universal cover $\mt$ of $M$. If we consider $\cL^s$ to be the quotient space $\mt/_{\wfws}$ by considering two points to be equivalent if they belong to the same leaf of $\wfws$ we have that $\cL^s$ is a one-dimensional simply connected manifold (which may not be Hausdorff). 
\end{lema}

\begin{proof}
This is an exercise which is left to the interested reader. (See also \cite[\S 4]{Calegari} or \cite{CamachoLins}.)
\end{proof}

One dimensional simply connected (not necessarily Hausdorff) manifolds can be characterized as the leafspaces of one-dimensional foliations on the plane $\RR^2$ (see \cite{CamachoLins}) and a result of Palmeira states that every foliation by planes in a simply connected 3-manifold is homeomorphic to a foliation of $\RR^2$ times $\RR$ (in particular, the manifold has to be $\RR^3$). Note that since $\mt$ is simply connected, the foliation $\wfws$ is transversally oriented and this induces an orientation in $\cL^s$ that allows us to speak about leaves \emph{above} and \emph{below} each other. Some remarks are in order: 

\begin{itemize}
\item Each leaf $L \in \wfws$ which is a point in $\cL^s$ separates $\cL^s$ in two connected components. It follows that $\mt \setminus L$ consists of two half spaces $L^+$ and $L^-$ whose boundary is $L$ and such that projected to $\cL^s$ partition $\cL^s$ into the leaves 'above' (those in the projection of $L^+$) and 'below' (those in the projection of $L^-$) the leaf $L$. 

\item We say that two leaves $L,F$ are \emph{nested} if $L^+ \subset F^+$ or $F^+ \subset L^+$. 

\item We say that two leaves $L, F$ are \emph{non-separated} if there is a sequence $L_n \in \cL^s$ such that $L_n$ converges to both $L$ and $F$ (and possibly other leaves). In this case, it is easy to see that $L$ and $F$ cannot be nested. We say that $L$ and $F$ are \emph{branching} leaves. The branching along $L$ is \emph{positive} if the sequence $L_n \in L^+$ and negative if $L_n \in L^-$. Note that if $L_n \in L^+$ then it also holds that $L_n \in F^+$. 

\item Note that if two leaves $L, F$ are not nested, it means that they must have some non-separated leaves in between. More precisely, not being nested means that either $L^+ \cap F^+$ or $L^- \cap F^-$ are empty, we consider the region in \emph{between} to be $(L^+ \cap F^+) \cup (L^- \cap F^-)$ (note that one of the sets in the union is empty), and what we claim is that there must be non-separated leaves $L', F'$ such that every curve in $\mt$ from $L$ to $F$ must intersect both $L'$ and $F'$. 

\item Note also that if two leaves $L$ and $F$ are nested (say, that $F^+ \subset L^+$), this does not imply that there is a (positive) transversal from $L$ to $F$. 
\end{itemize}

All the considerations made above, as well as Lemma \ref{lem.leafspace} hold for $\wfwu$ which also defines a leaf space $\cL^u$. 

Note that the fundamental group $\pi_1(M)$ of $M$ acts on $\cL^s$ and $\cL^u$ faithfully, with the following properties which are not so hard to check (but some are challenging exercises if one does not have background in hyperbolic dynamics, see \cite{KH, FH} for generalities on hyperbolic dynamics): 

\begin{itemize}
\item every $\gamma \in \pi_1(M) \setminus \{\mathrm{id}\}$ acting on $\cL^s$ and $\cL^u$ has a discrete set of fixed points which can be either attracting or repelling, in particular, each fixed point is isolated.

\item If $L \in \cL^s$ is fixed by $\gamma$ then, there is a unique $F \in \cL^u$ which is fixed by $\gamma$ such that $L \cap F \neq \emptyset$ in $\mt$, the intersection corresponds exactly to the lift of a periodic orbit in $M$ whose free homotopy class is the conjugacy class of $\gamma$. Moreover, if $\gamma$ is attracting on $L$, it has to be repelling on $F$ (and viceversa). The fact that it is attracting or repelling depends only on the flow orientation of the corresponding periodic orbit with respect to the action of $\gamma$ on it. 

\item The set of points $L \in \cL^s$ which are fixed by some deck transformation is dense. (Note that this is independent on whether the flow is transitive or not, it just depends on the density of the union of the stable manifolds of periodic points which is a consequence of the shadowing lemma for hyperbolic systems.) 

\item The flow $\phi$ is transitive (i.e. has a dense orbit) if and only if the action of $\pi_1(M)$ on $\cL^s$ (resp. $\cL^u$) is minimal. This is because otherwise, there would be a closed $\pi_1(M)$ invariant set of $\cL^s$ which would project to a compact $\cF^{ws}$-saturated (resp. $\cF^{wu}$-saturated) set in $M$ which produces a repeller (resp. attractor), contradicting the existence of a dense orbit. 
\end{itemize}

These actions sometimes impose restrictions on the algebraic structure of the groups that can admit them. We refer the reader to \cite{Barbot-1mfd, RSS, Fenley-1mfd} for results in these directions. In particular, in \cite{RSS} a class of hyperbolic 3-manifolds is shown not to admit Anosov flows. We note however that all these obstructions correspond to either the existence of a Reebless foliation on the manifold, or existence of a minimal foliation, but not really about existence of Anosov flows (some of them also depend on orientability considerations). Recently, in \cite{BinYu} an example of a hyperbolic 3-manifold admitting taut foliations but not Anosov flows was presented, this corresponds to rather particular hyperbolic 3-manifolds where there are enough tools to understand quite precisely all the possible foliations such a manifold admits. In any case, many very natural questions remain open: 

\begin{quest}
Is there a hyperbolic 3-manifold so that no finite cover admits an Anosov flow? 
\end{quest}

Note that by Agol's theorem \cite{Agol} every hyperbolic 3-manifold has a finite cover which admits a taut foliation (in fact, a foliation by compact surfaces). However, it is clear that such a foliation cannot be even deformed to be the weak stable foliation of an Anosov flow (indeed, the tangent space of such a foliation has non zero Euler class, so no deformation can admit a non-vanishing vector field tangent to it). Right now, we lack techniques to address this kind of question. 

\begin{quest}\label{quest3}
What are the possible topologies of the leaf spaces of Anosov flows? Are there examples where $\cL^s$ is not homeomorphic to $\cL^u$? 
\end{quest}

In the next section we will show that the second question admits a negative answer when $\cL^s \cong \RR$. It is also possible to produce some restrictions on the topology of leaf spaces, but as far as the author is aware, a complete answer is not clear. A very related question was asked in \cite[Remark 3.43]{Barthelme}. 

A related but quite different question is the following (I am not sure if the answer is known or not, probably some examples can be made, but I have not seen it explicitly answered in the literature): 

\begin{quest}\label{quest-fol}
Let $\phi_t$ be a transitive Anosov flow which is not a suspension, is it true that $\cF^{ws}$ is homeomorphic to $\cF^{wu}$?\footnote{Note that it is not hard to construct non-transitive counterexamples to this by using the ideas of \cite{FW,BBY} and making an example with two repellers and one attractor. Moreover, note that some suspensions verify that the weak stable and weak unstable cannot be mapped one to the other by a homeomorphism. Mario Shannon suggested some possible constructions starting from this example which could produce a negative answer to this question. In the conference, the related question of finding non-$\RR$-covered flows which are not orbit equivalence to their 'flip' was also discussed (see Question \ref{quest-flip}).} 
\end{quest}

Also, there is a very important conjecture (called the $L$-space conjecture) stating the equivalence between the existence of a cooriented taut foliation, the left orderability of the fundamental group, and not being an $L$-space. In this setting, one can ask the following question (which as far as I know is open)\footnote{Note that having an action on $S^1$ does not imply that it is left orderable, because it is not always possible to lift the action to the line without considering an extension of the group. If the $H_2(M)$ has positive rank, the group is left orderable, similarly if it is $0$, because the action can be lifted, but the case where $H_2(M)$ is torsion, is not clear.}:  

\begin{quest}
Is the fundamental group of a 3-manifold admitting an Anosov flow left orderable? 
\end{quest}

In particular, it could be that the study made in \cite{RSS,Fenley-1mfd} can be refined by restricting to leaf spaces which correspond to foliations of Anosov flows, but it seems we lack a bit of understanding to pursue this here. As suggested to me by Sergio Fenley, it could be that the fact that there is a well understand structure of the non-separation of leaves for Anosov foliations, could help in this problem (see \cite{Fenley-branch}).

\subsection{$\RR$-covered Anosov flows and transitivity}

When one of the leaf spaces $\cL^s$ or $\cL^u$ is Hausdorff (in which case, it must be homeomorphic to $\RR$) we can obtain certain dynamical consequences, for instance: 

\begin{prop}\label{prop-rcovtransitive}
Assume that $\phi_t: M \to M$ is an Anosov flow such that $\cL^s \cong \RR$, then, $\phi_t$ is transitive (i.e. it has a dense orbit). 
\end{prop}

\begin{proof}We will assume that everything is orientable for simplicity. 

If $\phi_t$ is not transitive, then, there is a closed $\pi_1(M)$ invariant set $\Lambda \subset \cL^s$ (corresponding to the lift of a hyperbolic repeller) so that $\cL^s \neq \Lambda$. 

Consider a connected component $I$ of the complement of $\Lambda$ in $\cL^s$. Let $L \in \cL^s$ be an extremity of $I$ and assume that $E \in I$ is a leaf which is fixed by some deck transformation $\gamma \in \pi_1(M)$. Since $\Lambda$ is $\pi_1(M)$ invariant, we deduce that $\gamma(L)=L$. Since the stabilizer of $L$ is cyclic, it follows that there is $\eta \in \pi_1(M)$ such that if an element of $\pi_1(M)$ has a fixed point in $I$, then it belongs to the group $\langle \eta \rangle$ generated by $\eta$. On the other hand, we know that the fixed point set of $\eta$ is discrete, and that the set of elements which are fixed by some element of $\pi_1(M)$ is dense in $\cL^s$, this is a contradiction. 
\end{proof}

There are many other reasons to consider this class as particular and give it a name: 

\begin{defi}\label{def-rcov}
An Anosov flow $\phi_t: M \to M$ is called $\RR$-\emph{covered} if $\cL^s$ or $\cL^u$ are homeomorphic to $\RR$. 
\end{defi}

It follows from the work of \cite{Fenley,Barbot-leaf} that if one of the leaf spaces is homeomorphic to $\RR$, then, so is the other. We will see this later, we first need to understand the orbit space of an Anosov flow. 

We will give an indication of a different proof of this fact, using a result of Brunella \cite{Brunella} that is interesting by itself: 

\begin{teo}[Brunella]\label{teo.brunella}
Let $\phi_t : M \to M$ be a non-transitive Anosov flow with orientable bundles. Then, there is a union of incompressible tori $T_1, \ldots, T_k$ which are transverse to the flow, their union separates $M$  and there are orbits staying in each component of $M \setminus \bigcup T_i$. 
\end{teo}

\begin{proof}[Sketch of the proof]
It is a general fact that flows on compact manifolds admit \emph{Lyapunov functions}. This means that there is a function $V: M \to \RR$ so that the gradient $\nabla V$ is never orthogonal to the vector field $X$ generating the Anosov flow. One can assume that this Lyapunov function is smooth. If the flow is not transitive, we get that the Lyapunov function is not constant, so, by considering a regular value, we obtain a compact submanifold $S$ everywhere transverse to the Anosov flow. This compact submanifold (which is not necessarily connected) separates $M$ in (at least) two connected components. Since each connected component of $S$ admits a foliation (the intersection with, say, the weak stable foliation) we deduce that each connected component of $S$ is a torus (recall we are assuming that everything is orientable). 

We need to show that each tori is incompressible to complete the proof. For this, we use the following fact about irreducible 3-manifolds: if a torus is compressible, then, it either bounds a solid torus or is contained in a ball (see e.g. \cite[Lemma A.2]{BFFP}). In both cases, it follows that either the forward or backward orbit of the torus remains in a manifold (with boundary) of subexponential growth, contradicting Theorem \ref{t.margulis} (or rather a local version of it, which admits the same proof). 
\end{proof}

Using this result, one can prove that a non-transitive Anosov flow cannot be $\RR$-covered by showing the following result that we will deduce in the next section: 

\begin{prop}\label{prop-transversetorus} 
Let $\phi_t: M \to M$ be an Anosov flow and let $T$ be a transverse torus so that there are orbits of $\phi_t$ which do not intersect $T$. Then, $\phi_t$ is not $\RR$-covered.   
\end{prop}

Note that in \cite{FW} examples of non-transitive Anosov flows had been constructed. For a while, it was thought that maybe transitive Anosov flows could all be $\RR$-covered, but this was disproved (using the previous proposition) by a beautiful example from \cite{BL}. Now, there is a zoo of examples \cite{BBY} that use the previous proposition. In \cite{Fenley-branch} the first example of non-$\RR$-covered Anosov flows without transverse tori was constructed. Again, now there is a large industry, see \cite{BI}.

\subsection{Orbit space and the bifoliated plane} 

Here, we introduce, for an Anosov flow $\phi_t: M \to M$, its orbit space $\cO_\phi = \mt /_\sim$ where we consider $y \sim x$ if there exists $t \in \RR$ so that $y = \tilde \phi_t(x)$ (note that this is an equivalence relation). 

The main result of this subsection is that

\begin{teo}[Barbot-Fenley]\label{teo-bifoliated}
The set $\cO_\phi$ is homeomorphic to $\RR^2$ and it is \emph{bifoliated} by two transverse foliations $\cG^s$ and $\cG^u$ which correspond to the projections to the quotient of the foliations $\wfws$ and $\wfwu$ respectively. The action of the fundamental group $\pi_1(M)$ induces an action on this bifoliated plane. 
\end{teo}

\begin{proof} 
We follow the proof from \cite{FP-Hsdff} which is slightly more general (it holds for general transverse foliations on which the leaf space of the intersected foliation in each leaf is Hausdorff). A key fact, which follows from Theorem \ref{teo.stableman} is that in $\mt$ if one considers a leaf $L \in \wfws$ and $E \in \wfwu$, the intersection $L\cap F$ if non empty, is exactly one orbit of $\tilde \phi_t$. See figure \ref{f.twocomp}.

\begin{figure}[ht]
\begin{center}
\includegraphics[scale=0.65]{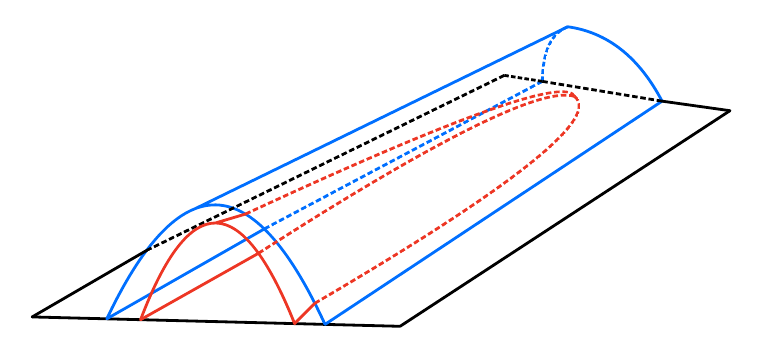}
\begin{picture}(0,0)
\end{picture}
\end{center}
\vspace{-0.5cm}
\caption{{\small If a pair of leaves intersects in more than one component, one can see the non-Hausdorffness inside each of the leaves.}}\label{f.twocomp}
\end{figure}

We consider a sequence of orbits $\ell_n \in \cO_\phi$ with points $p_n, q_n \in \mt$ (so that $p_n =\tilde \phi_{t_n}(q_n)$) such that $p_n \to p$ and $q_n \to q$. We want to show that $p$ and $q$ belong to the same orbit of $\tilde \phi_t$. 

Without loss of generality and up to subsequence, we can assume that in small
foliated neighborhoods of $p$ and $q$ respectively, the points $p_n$ and $q_n$ are weakly
monotonic in the leaf space of each of the foliations (in fact, one can choose them
to either be strictly monotonic, or to stay always in the same plaque as $p$ or $q$).
Let $L_n \in \wfws$ and $E_n \in \wfwu$ be so that $\ell_n = L_n \cap E_n$ (recall that weak stable and weak unstable leaves in $\mt$ intersect in a unique orbit). 

Let $L_p, L_q \in \wfws$ be the leaves through $p$ and $q$ respectively and $E_p,E_q \in \wfwu$ the corresponding leaves of $p$ and $q$. We want to show that $L_p=L_q$ and $E_p=E_q$ which concludes thanks to the key remark above. 

We define $c_{n,k} =  E_n \cap L_k$. Note that it is non-empty since $p_n$ and $p_k$ both belong
to a foliated box close to $p$. 

We can consider sequences $x_n \to p$ with $x_n \in L_p \cap E_n$ and similarly $y_n \to q$ with
$y_n \in L_q \cap E_n$. Fixing $n>0$ so that $x_n$ is very close to $p$ we get that we can find
points $z_k \to x_n$ and $w_k \to  y_n$ in the orbit $c_{n,k}$. Since the points $z_k$ and $w_k$ are in the same orbit for all $k$ and all contained in $E_n$, it follows that $x_n$ and $y_n$ must be in the same orbit, because inside $E_n$ there are no non-separated orbits. In particular this implies that $L_p = L_q$. A symmetric argument shows that $E_p =E_q$ and concludes the proof that $\cO_\phi$ is Hausdorff. Since it is a topological surface, non compact and it is simply connected, it follows that it has to be the plane. 

Now, defining the foliations $\cG^s$ and $\cG^u$ as the projections of $\wfws$ and $\wfwu$ we get the bifoliated plane. Clearly, since all objects are invariant under deck transformation, this structure admits the desired action of $\pi_1(M)$. 

\end{proof}

\begin{obs}
It is a fact that for smooth 3-dimensional Anosov flows the weak stable and unstable foliations are in fact $C^{1}$ (and in fact, with H\"{o}lder continuous derivatives). Since the flow is smooth, we get smooth charts, so that the bifoliated plane can be given a smooth structure on which $\pi_1(M)$ acts. In particular, one can get better regularity of this action. 
\end{obs}

The bifoliated plane allows a reduction of dimension in the understanding of Anosov flows: from a flow in 3-dimensions, we get a group action on a plane with two transverse foliations (which behave somewhat as 1-dimensional objects). For this to actually make sense, one needs the following result which essentially says that up to orbit equivalence, the bifoliated plane does not loose any information (see \cite{Barbot-1mfd} for a proof): 

\begin{teo}\label{teo-bifolclas}
Assume that $\phi^1_t, \phi^2_t: M \to M$ are Anosov flows with bifoliated planes $(\cO_1, \cG^s_1, \cG^u_1)$ and $(\cO_2, \cG^s_2, \cG^u_2)$. Then, $\phi^1_t$ and $\phi^2_t$ are orbitally equivalent by an orbit equivalence homotopic to identity if and only if there is a $\pi_1(M)$-equivariant homeomorphism $H: \cO_1 \to \cO_2$ (i.e. we have that for every $\gamma \in \pi_1(M)$ and $x \in \cO_1$ it holds $H(\gamma x)= \gamma H(x)$). 
\end{teo}

Note that one can also characterize general (not necessarily homotopic to identity) orbit equivalences by letting some automorphism of $\pi_1(M)$ act on the equivariace condition. See \cite{BFrM} for more details.

\begin{proof}[Sketch of the proof]

The key point is to produce an homotopy equivalence of $M$ (homotopic to identity) which maps orbits of $\phi^1_t$ to orbits of $\phi_t^2$. Once this is obtained, there is a classical averaging trick (see \cite{Verjovsky,Ghys,Barbot-1mfd}) that allows to promote this homotopy equivalence to an actual homeomorphism which gives the desired orbit equivalence. The existence of the homotopy equivalence is given by a general result of Heafliger \cite{Hae}. 

\end{proof}

\begin{figure}[ht]
\begin{center}
\includegraphics[scale=0.55]{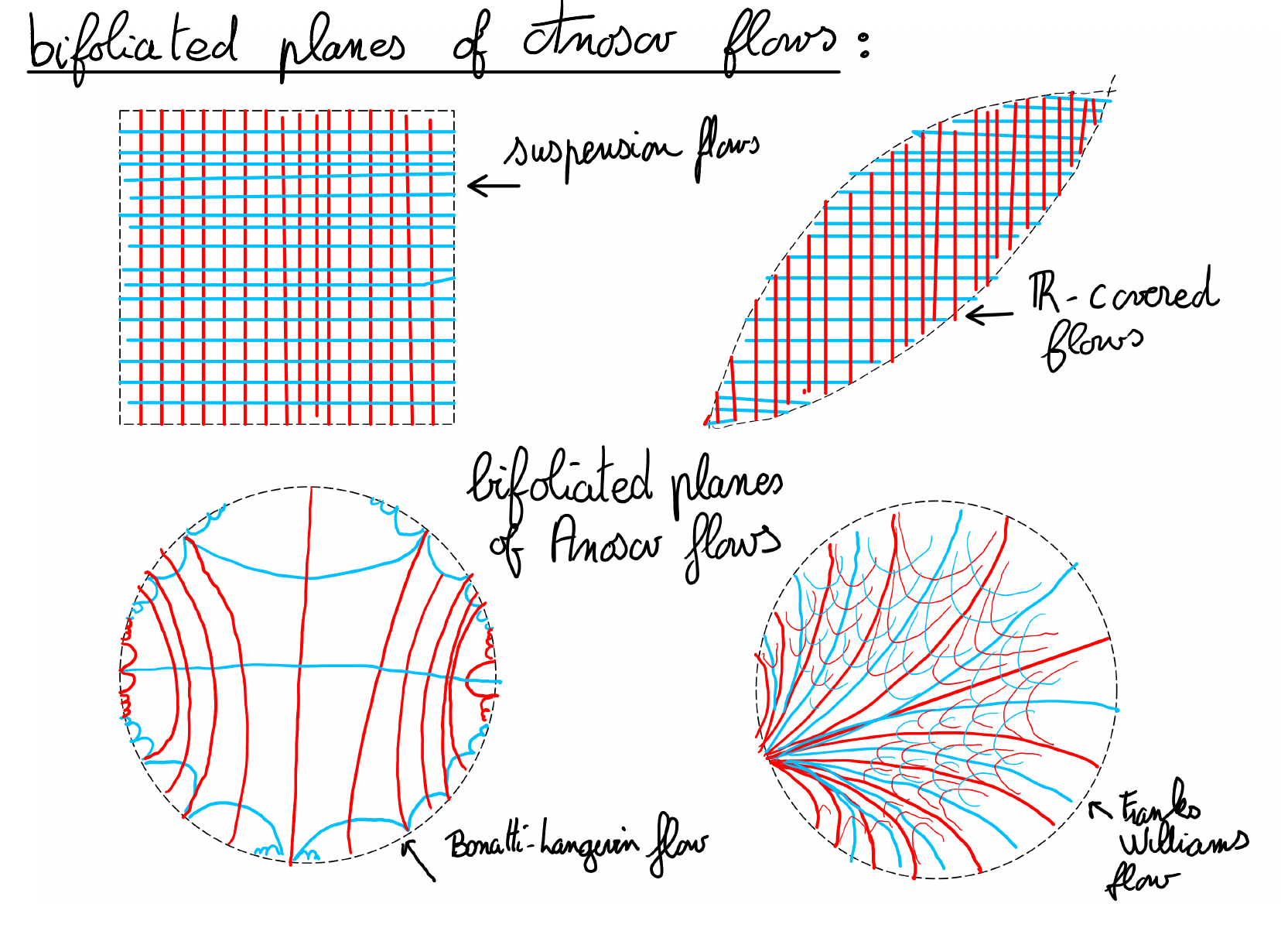}
\begin{picture}(0,0)
\end{picture}
\end{center}
\vspace{-0.5cm}
\caption{{\small The bifoliated plane for some examples. (Figure by Christian Bonatti.)}\label{f.bifol}}
\end{figure}

We get some simple properties of the action that we sumarize here (we refer the reader to \cite{BFrM,BBM} for an abstraction of actions in a bifoliated plane):  

\begin{prop}\label{prop-propbifol}
The action of $\pi_1(M)$ in $(\cO_\phi, \cG^s, \cG^u)$ has the following properties: 
\begin{itemize}
\item If $\gamma \in \pi_1(M) \setminus \{\mathrm{id}\}$ verifies that $\gamma \ell = \ell$ for some $\ell \in \cG^s$, then, there is (a unique) $x \in \ell$ so that $\gamma x = x$. Moreover, if $\ell' \in \cG^u$ is the leaf through the point $x$ the action of $\gamma$ is contacting in one of $\ell, \ell'$ and repelling on the other (it is a saddle point). 
\item If $\phi_t$ is transitive, then, the set of points fixed by some non trivial element of $\pi_1(M)$ is dense in $\cO_\phi$. 
\end{itemize}
\end{prop}

\begin{proof}
These are properties which follow mostly from the dynamical properties of Anosov flows (c.f. \cite{FH}). We leave this as an exercise. 
\end{proof}

Now we are ready to show: 

\begin{proof}[Sketch of the proof of Proposition \ref{prop-transversetorus}] 
The transverse torus lifts to a plane transverse to the orbits of the flow due to the fact that it has to be incompressible (recall Theorem \ref{teo.brunella}) and it therefore can be embedded as a plane $\hat T$ inside the orbit space $\cO_\phi$. By the fact that not every orbit intersects the torus, we know that $\hat T \neq \cO_\phi$. If inside $\hat T$ we have non-separated leaves of $\cG^s$ or $\cG^u$ we are done. Else, one can show that inside $\hat T$ the foliations have a global product structure (because transverse foliations in tori without Reeb-annuli must have global product structure). Since the plane $\hat T$ does not cover the full orbit space, one can consider a sequence $\ell_n$ of leaves of, say $\cG^u$ escaping to infinity. These leaves cannot converge to a unique leaf, otherwise, the global product structure of $\hat T$ could be extended further to a maximal plane invariant under the $\ZZ^2$ which preserves $\hat T$. This shows that the leaf space cannot be Hausdorff. 
\end{proof}

Recently, some understanding on more precise effect of the existence of transverse tori in the bifoliated plane has been gained, but as far as I understand, there are still some things to understand further. See \cite{BFeM,BBM}. 

\subsection{Dichotomy for $\RR$-covered Anosov flows} 

The following beautiful result can only be followed with some pen and paper at hand. 

\begin{teo}[Barbot-Fenley]\label{teo.rcoveredboth}
Let $\phi_t: M \to M$ be an Anosov flow. Then, one has that $\cL^s \cong \RR$ if and only if $\cL^u \cong \RR$.
\end{teo}

\begin{proof}
Let us assume that $\cL^s \cong \RR$. Note that by Proposition \ref{prop-rcovtransitive} implies that the action of $\pi_1(M)$ on $\cL^s$ is minimal.

We have shown in Theorem \ref{teo-bifoliated} that $\cO_\phi$ is a bifoliated plane, so, up to a change of coordinates (homeomorphism of the plane), we can assume that $\cG^s$ (i.e. the projection of $\cL^s$ to $\cO_\phi$) is the foliation by vertical lines $\cG^s = \{x = t\}_{t\in \RR}$. Note also that the transtivity of $\phi_t$ also implies density of periodic orbits by the classical Anosov closing lemma, so, we get that the set of points in $\cO_\phi$ which are fixed by some deck transformation is dense in $\cO_\phi$.  In these coordinates, we have that every leaf $\ell \in \cG^u$ is the graph of a function from an open interval (possibly unbounded) $I_\ell = (a_\ell, b_\ell)$ (where $a_\ell \in \RR \cup \{-\infty\}$ and $b_\ell \in \RR \cup \{+\infty\}$) and $I_\ell$ corresponds in $\cL^s$ to the set of leaves that $\ell$ intersects. We will denote by $f_\ell: I_\ell \to \RR$ the function of which $\ell$ is the graph. Note that if $a_\ell$ or $b_\ell$ are finite, then, the function $f_\ell$ must tend to $\pm \infty$ when $x \to a_\ell, b_\ell$.   

We will freely use the properties from Proposition \ref{prop-propbifol} without making reference to them. 

{\bf Case 1: } Consider first the case where there is a leaf $\ell \in \cG^u$ whose projection in the second coordinate is unbounded (i.e. it intersects every leaf in a ray of $\cL^s$, or equivalently, that either $a_\ell = -\infty$ or $b_\ell = \infty$). We will show that in this case, every leaf $\ell' \in \cG^u$ intersects every leaf of $\cG^s$, in particular, leaves of $\cG^s$ intersect every leaf of $\cG^u$ which implies that $\cG^u$ is $\RR$-covered as desired. 

To see this let us assume without loss of generality that $\ell \subset \cO_\phi \cong \RR^2$ verifies that $b_\ell = +\infty$. Note that the set of leaves with this property is invariant under the action of $\pi_1(M)$, so, given $\ell' \in \cG^u$ we can find sequences $\gamma_n, \eta_n \in \pi_1(M)$ so that $\gamma_n \ell \to \ell'$ and $\eta_n \ell \to \ell'$ from the left and the right (recall that every leaf $\ell' \in \cG^u$ separates $\cO_\phi$ in two components and there is an induced orientation), it means that the graphs of $\ell_n = \gamma_n \ell$ and $\hat \ell_n = \eta_n \ell$ are above and below the graph of $\ell'$, or equivalently, that in their domains, one has that $f_{\hat \ell_n} < f_{\ell'} < f_{\ell_n}$. Since $b_{\hat \ell_n} = b_{\ell_n}= + \infty$ we have that in any compact interval, the value of $f_{\ell'}$ is bounded and therefore, $b_{\ell'}$ must also be $+\infty$. 

This implies that for every $\ell' \in \cG^u$ we have that $b_{\ell'}=\infty$. Now, let us show that there must be some leaf $\ell' \in \cG^u$ so that $a_{\ell'}= -\infty$, then, applying the same argument we will deduce that every leaf $\ell \in \cG^u$ verifies that $I_\ell = \RR$ and this gives that every leaf of $\cG^s$ intersects every leaf of $\cG^u$ (and thus $\cG^u$ is $\RR$-covered). 

We can thus assume by contradiction that every leaf $\ell$ of $\cG^u$ verifies $a_\ell \in \RR$. Note that one has that $a_{\gamma \ell} = \gamma a_\ell$ (here, we are considering $\gamma a_\ell$ to be a leaf in $\cL^s$ considering $a_\ell$ not only as a real number but also as a leaf of $\cG^s$). In particular, there are leaves $\ell_n$ so that $a_{\ell_n} \to -\infty$.  Consider an element $x \in \cO_\phi$ such that $\gamma x = x$ for some $\gamma \in \pi_1(M) \setminus \{\mathrm{id}\}$. If $x \in \ell \in \cG^u$ we have that $a_\ell$ (as a leaf of $\cG^s$) is fixed by $\gamma$ and thus also has a fixed point $y$ which belongs to a leaf $\ell' \in \cG^u$. Note that by assumption, we have that $b_{\ell'} = \infty$ and thus, this implies that $\ell'$ intersects the leaf of $\cG^s$ containing $x$, and since the unstable of $x$ does not intersect $a_\ell$ we get that the intersection takes place in a point different from $x$, which produces two fixed points in $\ell'$ for $\gamma$ a contradiction. This completes the proof in case 1. 

\smallskip

{\bf Case 2:} The functions $a_\ell$ and $b_{\ell}$ are bounded for every $\ell$. 

We first assume that there is a leaf $\ell \in \cG^u$ such that $f_\ell(c)$ tends to $+\infty$ as $c \to b_\ell$ and $f_\ell(c)$ tends to $-\infty$ as $c \to a_\ell$. We will give a name to this property and say that $\ell$ is \emph{increasing} (the case where there is some $\ell$ which is \emph{decreasing} is symmetric). In this case, we claim that every $\ell' \in \cG^u$ is increasing. To show this, we need to exclude the case where $f_{\ell'} (c) \to -\infty$ as $c \to b_{\ell'}$ (a symmetric argument will exclude the possibility that $f_{\ell'}(c) \to + \infty$ as $c \to a_{\ell'}$). 

To do this, we use the minimality of the action of $\pi_1(M)$ on the stable leaf space $\cL^s$ and the fact that we can assume that the action is orientation preserving (maybe after considering a finite index subgroup) so, for all $\gamma \in \pi_1(M)$ the image $\gamma \ell$ of $\ell$ is also increasing. Now, using the minimality of the action of $\pi_1(M)$ we can choose $\gamma$ so that $\gamma a_\ell \in (a_{\ell'}, b_{\ell'})$ but since $\gamma \ell$ is increasing, and  $f_{\ell'} (c) \to -\infty$ as $c \to b_{\ell'}$ we get that $\gamma \ell \cap \ell \neq \emptyset$ a contradiction. 

Now, we need to exclude the case where no leaf is increasing or decreasing, that is, every leaf $\ell$ verifies that either $f_{\ell}(c) \to + \infty$ as $c \to a_\ell$ and as $c\to b_\ell$ or $f_\ell(c) \to - \infty$ as $c \to a_\ell$ and as $c\to b_\ell$. Without loss of generality, assume that there is a leaf $\ell$ such that $f_{\ell}(c) \to + \infty$ as $c \to a_\ell$ and as $c\to b_\ell$, we will say that $\ell$ is $U$-\emph{shaped}.  

Using density of periodic points, we can find point $p$ so that $\gamma p = p$ for some $\gamma \in \pi_1(M) \setminus \{\mathrm{id}\}$ and so that $\ell_p = \cG^u(p)$ is \emph{above} $\ell$. Being above means that $a_\ell < c_p=\cG^s(p) < b_\ell$ and $f_{\ell_p}(c_p) > f_\ell (c_p)$.

\begin{figure}[ht]
\begin{center}
\includegraphics[scale=0.85]{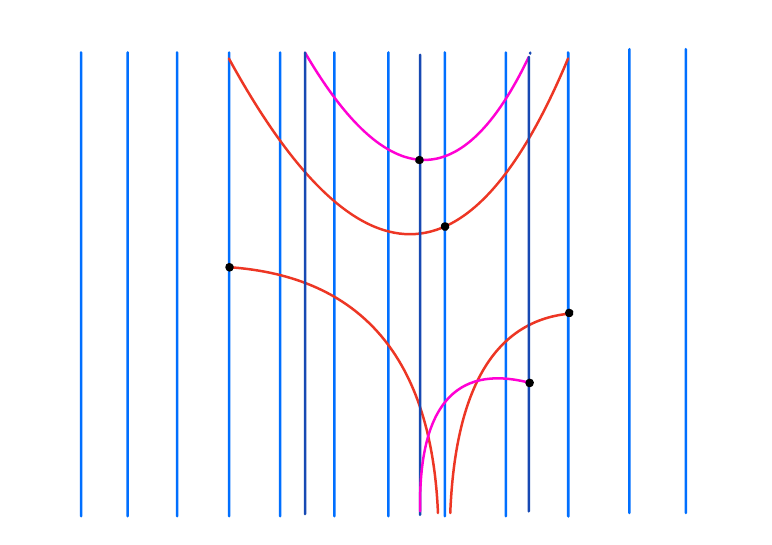}
\begin{picture}(0,0)
\put(-132,125){{\small $p =\gamma p$}}
\put(-85,190){ $\ell_p$}
\put(-152,165){{\small $q =\gamma' q$}}
\put(-105,190){$\ell_q$}
\end{picture}
\end{center}
\vspace{-0.5cm}
\caption{{\small Depiction of the argument in case 2}\label{f.rcov}}
\end{figure}
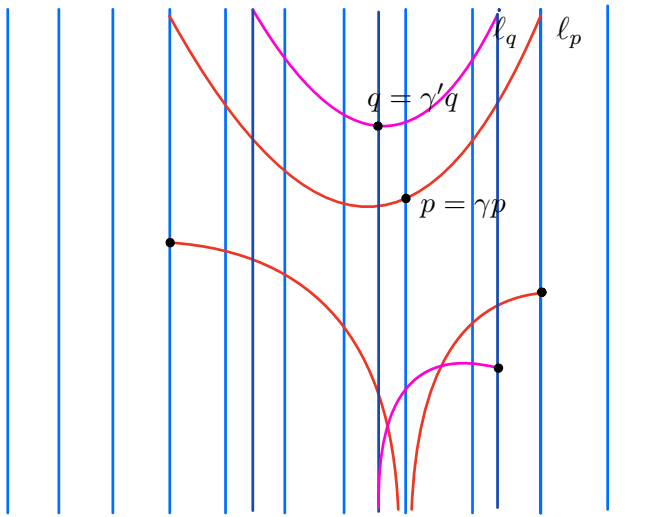

Note that $\gamma \ell_p=\ell_p$ so $\gamma a_{\ell_p}= a_{\ell_p}$ and $\gamma b_{\ell_p}= b_{\ell_p}$. In particular there are points $p_1, p_2$ which are fixed by $\gamma$ in the stable leaves $a_{\ell_p}$ and $b_{\ell_p}$ respectively. We claim that if $\ell_1 = \cG^u(p_1)$ and $\ell_2= \cG^u(p_2)$ we have that $b_{\ell_1} = c_p$ and $a_{\ell_2}= c_p$. To see this, assume that $b_{\ell_p} < b_{\ell_1}< c_p$ (the other case is symmetric) and note that $\gamma b_{\ell_1}= b_{\ell_1}$. But this forces $\ell_p$ to have two distinct fixed points, a contradiction. 

Also, note that $f_{\ell_1}(c) \to -\infty$ as $c \to c_p=b_{\ell_1}$, this is because the point $p_1$ cannot be above $\ell_p$ and thus if $f_{\ell_1}(c) \to +\infty$ as $c \to c_p=b_{\ell_1}$ this would force that $\ell_1 \cap \ell_p \neq \emptyset$. The symmetric argument gives $f_{\ell_2}(c) \to -\infty$ as $c \to c_p=a_{\ell_2}$. 

Now, we use the density of periodic points again to produce a point $q$ above $\ell_p$ which is fixed by some deck transformation $\gamma' \in \pi_1(M)\setminus \{\mathrm{id}\}$. Denote $\ell_q=\cG^u(q)$, $c_q=\cG^s(q)$ as above and we have that again $\ell_q$ is $U$-shaped and that $a_{\ell_p} < a_{\ell_q} < b_{\ell_q} < b_{\ell_p}$. Note also that without loss of generality we can assume that $c_p < c_q$ (note that the strict inequalities come from the fact that $\gamma'$ cannot be in the same cyclic group as $\gamma$ as otherwise some leaf of $\cG^s$ or $\cG^u$ would have more than one fixed point for some non-trivial deck transformation). 

Denote by $q'$  the fixed point of $\gamma'$ in the stable leaf $a_{\ell_q}$. By the argument made above, we know that if we denote $\ell' = \cG^u(q')$ then $b_{\ell'}=c_q$ and that $f_{\ell'}(c) \to -\infty$ as $c \to b_{\ell'}$. This forces $\ell'$ to intersect $\ell_2$ which is a contradiction. This completes the proof.  

\end{proof}

Among $\RR$-covered Anosov flows, there are two classes, product Anosov flows (which correspond to suspensions) and \emph{skewed}-$\RR$-\emph{covered} Anosov flows which we now describe: 

\begin{defi}
An Anosov flow is \emph{skewed}-$\RR$-\emph{covered} if $(\cO_\phi, \cG^s, \cG^u)$ is homeomorphic to the band $\cB=\{(x,y) \in \RR^2 \ : \ x-1 < y < x+1\}$ with the foliations by horizontal and vertical lines. 
\end{defi}

In the proof of Theorem \ref{teo.rcoveredboth} we essentially saw that if $\phi_t$ is an $\RR$-covered Anosov flow, then, it either is skewed $\RR$-covered, or, we have that every pair of leaves of $\cG^s$ and $\cG^u$ have a non empty intersection. It is a consequence of a result of Barbot (see \cite[Theorem 2.7]{Barbot-leaf}) that this implies that the flow is a suspension. 

\subsection{Lozenges} 

There are some configurations in the bifoliated plane that deserve a name. These objects have been introduced in \cite{Fenley} and turned out to be crucial to the study of Anosov flows. 

For a point $x \in \cO_\phi$ and $\ell = \cG^s(x)$ (resp. $\ell = \cG^u(x)$) we define a \emph{half leaf} to be the closure of a connected component of $\ell \setminus \{x\}$. We can say \emph{positive} or \emph{negative} half leaf if we consider a given orientation on the foliations. A \emph{quadrant} for a point $x \in \cO_\phi$ is given by two half leaves $\ell_1, \ell_2$ of $\cG^s(x)$ and $\cG^u(x)$ respectively and consists of the points in the connected component of $\cO_\phi \setminus (\ell_1 \cup \ell_2)$ not containing any other half leaf. The boundaries of the quadrant $Q_x$ of $x$ are the half leaves in the boundary, sometimes, we will refer to the stable and unstable boundary to denote the half leaf in $\cG^s(x)$ and $\cG^u(x)$ respectively. 

\begin{defi}[Lozenges]
Given an Anosov flow $\phi_t : M \to M$ an open set $\cD$ of $\cO_\phi$ is called a \emph{lozenge} with corners $x,y \in \cO_\phi$ if $\cD$ is the intersection of quadrants $Q_x$ and $Q_y$ respectively which verify that:
\begin{itemize}
\item the boundaries of the quadrants $Q_x$ and $Q_y$ are disjoint (called the \emph{sides} of the lozenge $\cD$),
\item for every $z \in \cD$ one has that $\cG^s(x)$ intersects the unstable sides of both $Q_x$ and $Q_y$, and $\cG^u(x)$ intersects the stable sides of both $Q_x$ and $Q_y$.
\end{itemize}
\end{defi}

If we denote the sides of $\cD$ as $A \subset \cG^s(x)$, $B \subset \cG^u(x)$, $C \subset \cG^s(y)$ and $D \subset \cG^u(y)$ we have the following properties (sometimes referred to as existence of \emph{perfect fits}): 

\begin{itemize}
\item $A \cap D = \emptyset$ and $B \cap C = \emptyset$, 
\item for any $\ell \in \cL^s$ one has that $\ell \cap B \neq \emptyset$ if and only if $\ell \cap D \neq \emptyset$,
\item for any $\ell \in \cL^u$ one has that $\ell \cap A \neq \emptyset$ if and only if $\ell \cap C \neq \emptyset$.
\end{itemize}

\begin{figure}[ht]
\begin{center}
\includegraphics[scale=0.55]{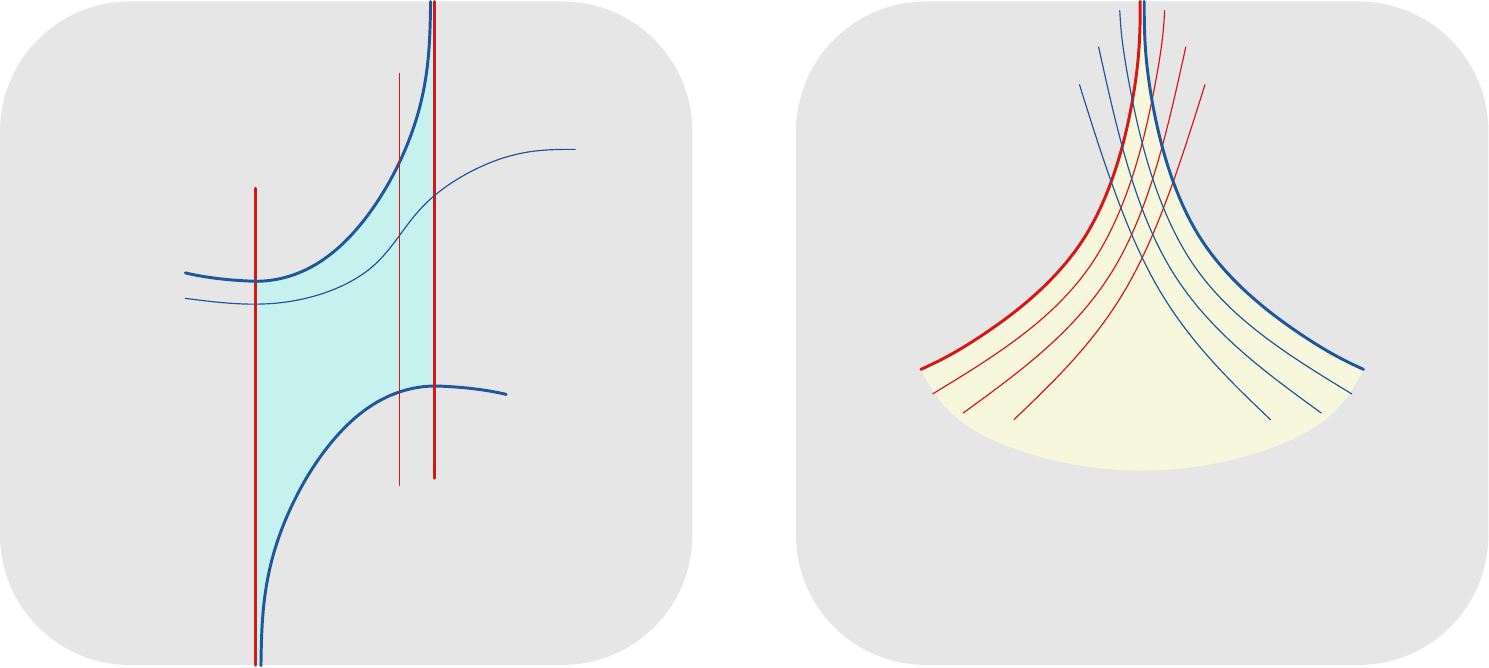}
\begin{picture}(0,0)
\end{picture}
\end{center}
\vspace{-0.5cm}
\caption{{\small On the left, a lozenge. On the right, a perfect fit. (Figure by Theo Marty)}\label{f.lozenge}}
\end{figure}

\begin{obs}
In the proof of Theorem \ref{teo.rcoveredboth} we exploited several times the fact that if a curve $\ell$ of say $\cG^u$ makes a \emph{perfect fit} with some curve $\ell'$ of $\cG^s$, then, if $\ell$ is fixed by $\gamma$ then so is $\ell'$ (under some orientability assumptions). 
\end{obs}

Recall that in the skewed-$\RR$ covered case, one has that one can identify $(\cO_\phi, \cG^s, \cG^u)$ with the band $\cB=\{(x,y) \in \RR^2 \ : \ x-1 < y < x+1\}$ with the foliations by horizontal and vertical lines. In this case one can easily see that every point is the corner of a lozenge. On the other hand, for a suspension, there is no lozenge at all. In the next section we shall see that non-$\RR$-covered Anosov flows always have lozenges, thus: 

\begin{prop}\label{prop-periodiccornersnonsusp}
An Anosov flow (with orientable foliations) which is not a suspension, always has a lozenge with periodic corners (i.e. fixed by the same non-trivial deck transformation). 
\end{prop}

Note that if a lozenge is fixed by some $\gamma \in \pi_1(M) \setminus \{\mathrm{id}\}$ then the corners correspond to periodic orbits of $\phi_t$ which are (freely) homotopic but the homotopy reverses the orientation of the flow. In this sense, the previous proposition can be understood as saying that an Anosov flow which is not a suspension always has a pair of periodic orbits which are freely homotopic to the inverse of the other. 

\begin{figure}[ht]
\begin{center}
\includegraphics[scale=0.90]{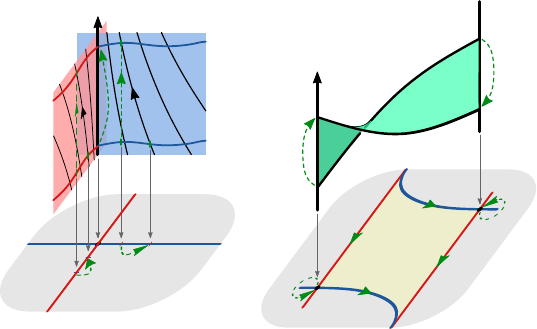}
\begin{picture}(0,0)
\end{picture}
\end{center}
\vspace{-0.5cm}
\caption{{\small Lozenges provide freely homotopic orbits going in oposite directions. In the left the lift of the flow over the leaf space is depicted. In the right, one can see how to lift a lozenge, the periodic orbits in the corners need to flow in oposite directions, so, one can put a band transverse to the flow and whose boundary glues with its image under the deck transformation, resulting in an annulus containing the periodic orbits in oposite directions and transverse to the flow in the interior. (Figure by Theo Marty)}\label{f.freelyhom}}
\end{figure}

\begin{quest}
Is there a \emph{contact topology} proof of this fact? (In the sense of \cite{Mitsumatsu,Hoozori,CLMM}, for instance.)
\end{quest}

An important property is that whenever an Anosov flow has two freely homotopic periodic orbits (without paying attention to orientation), then, there is a \emph{chain of lozenges} joining them. Let us be more precise. We will say that two lozenges $\cD,\cD'$ are \emph{consecutive} if their closures intersect, this implies that they share a corner. There are two possibilities, either they share a side (sometimes said, \emph{adjacent}), or they do not. 

\begin{defi}\label{def-chain}
A \emph{chain of lozenges} is a (finite or infinite) union of lozenges $\cC$, so that for any pair of lozenges $\cD, \cD'$ in the chain, there is a sequence of lozenges $\cD_1, \ldots, \cD_k$ so that $\cD_i$ is consecutive to $\cD_{i+1}$ and $\cD=\cD_1, \cD'=\cD_k$. If there are no adjacent lozenges in a chain of lozenges, we call the chain a \emph{string of lozenges}. 
\end{defi}

\begin{figure}[ht]
\begin{center}
\includegraphics[scale=0.70]{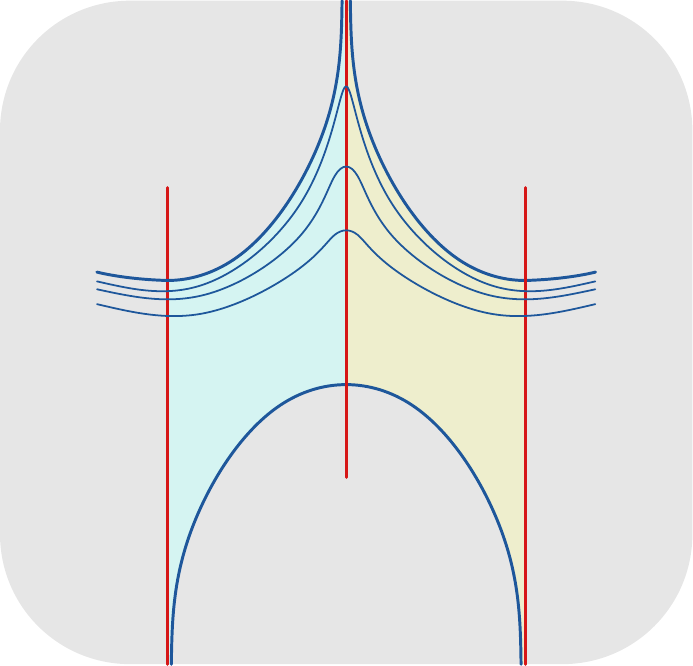}
\begin{picture}(0,0)
\end{picture}
\end{center}
\vspace{-0.5cm}
\caption{{\small Adjacent lozenges. (Figure by Theo Marty)}\label{f.adjloz}}
\end{figure}

Note that if two lozenges are adjacent, then the leaf space cannot be $\RR$-covered (see Figure \ref{f.adjloz}), since the leaves making a perfect fit with the side that contains both verifies that are not separated. In particular, the fact that two consecutive lozenges share a side can be characterized by the fact that there exists a leaf (of either $\cG^s$ or $\cG^u$) that intersects the interior of both lozenges. In the next section we will see that the converse also holds, whenever the Anosov flow is non $\RR$-covered, there is a pair of adjacent lozenges. This then actually characterizes non-$\RR$-covered Anosov flows. To show this, we will show that: 

\begin{prop}\label{prop-chainofloz}
If $\phi_t: M \to M$ is an Anosov flow with all bundles orientable, and $o_1, o_2$ are distinct periodic orbits which are freely homotopic (disregarding orientation), then, there are points $p_1,p_2 \in \cO_\phi$ (corresponding to lifts of $o_1$ and $o_2$ respectively) which are corners of a chain of lozenges which are all fixed by a deck transformation $\gamma \in \pi_1(M) \setminus \{\mathrm{id}\}$. 
\end{prop}

Note that the previous proposition is relatively direct in the $\RR$-covered case. If $\phi_t$ is a suspension, then, there are no pairs of distinct freely homotopic periodic orbits. When $\phi_t$ is skewed-$\RR$-covered, then, we leave this as an exercise for the reader (suggesting to use the one-step up maps, see \S \ref{ss.onestepup}). 

Lozenges are an important concept in the study of leaf spaces of Anosov flows in dimension 3 as we will see. Just to add an element, which is important in the study of the contact Anosov property (see \cite{Marty2}), is that lozenges allow one to produce annuli transverse to the flow whose boundaries consist on freely homotopic orbits turning in oposite directions (called \emph{partial Birkhoff sections} in \cite{Marty2}). The use of such annuli goes back at least to Barbot's work (see \cite{Barbot}).  See figure \ref{f.partialB}.

\begin{figure}[ht]
\begin{center}
\includegraphics[scale=1.65]{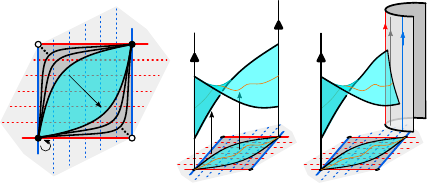}
\begin{picture}(0,0)
\end{picture}
\end{center}
\vspace{-0.5cm}
\caption{{\small Partial sections from a lozenge. (Figure by Theo Marty)}\label{f.partialB}}
\end{figure}

\subsection{Some properties of non-$\RR$-covered Anosov flows}\label{s.nonrcov}

The first important result about non-$\RR$-covered Anosov flows (besides their existence!) is the following \cite{Fenley-branch}: 

\begin{teo}\label{teobranching}
Let $\phi_t: M \to M$ be an Anosov flow. Let $L \in \widetilde{\cF^{ws}}$ be a leaf which is non-separated from some other leaf (in particular, $\phi_t$ is non-$\RR$-covered). Then, there exists $\gamma \in \pi_1(M)\setminus \{\mathrm{id}\}$ such that it fixes $L$ and all leaves non separated from it. 
\end{teo}

A feature of this proof is that it cannot be done purely in the bifoliated plane: we will need to project to $M$ and argue some parts there. We will try to do so the least as possible.  

\begin{proof}[Sketch of the proof] 
Consider $L, L'$ two non-separated leaves of $\widetilde{\cF^{ws}}$. We will just prove that $L$ is invariant under some non-trivial deck transformation. The fact that both $L$ and $L'$ are fixed by the same deck transformation is more complicated, we will indicate some of the ideas (sufficient for our purposes) in the sketch of the proof of the next result. 

Let $\cI^u(L) = \{F \in \widetilde{\cF^{wu}} \ : \ F \cap L \neq \emptyset \}$ and $\cI^u(L')$ defined similarly. Note that $\cI^u(L), \cI^u(L')$ are open by transversality and are disjoint. The disjointness can be more easily seen in the bifoliated plane: if a leaf $F \in   \widetilde{\cF^{wu}} $ intersects both $L$ and $L'$, then, taking a sequence $L_n$ converging both to $L$ and $L'$ we get that for large $n$ we have that $L_n \cap F$ has to intersect in two distinct orbits, which would made the orbit space restricted to $F$ non Hausdorff which is not the case. 

In particular, one can consider a leaf $F_0 \in \partial \cI^u(L)$ to be the unique leaf in the closure of $\cI^u(L)$ separating\footnote{Unfortunately, the notion of 'non-separated' leaves in a given foliation can produce confusion when using this term that is somewhat different. Two leaves in the same foliation are non-separated if there is a sequence of leaves converging to both, while here we are considering a leaf of the other foliation and saying that it separates two leaves in the sense we explain. Note however, that it is true that if $L$ is non-separated from $L'$, then, there is no leaf $L''$ so that $\mt \setminus L''$ contains $L$ and $L'$ in different connected components.} $L$ from $L'$ in the sense that $L$ and $L'$ are in different connected components of $\mt \setminus F_0$. 

Note that $L$ and $F_0$ make a perfect fit. We will show that $F_0$ is fixed by some non-trivial deck transformation (and this will force $L$ to be fixed too, because they make a perfect fit). 

Fix some $L_0$ close to both $L$ and $L'$ (which are non-separated) so that it intersects both $F_0$ and weak unstable leaves $F$ and $F'$ intersecting  $L$ and $L'$ respectively. Consider the orbits $x = L_0 \cap F$, $x' = L_0 \cap F'$ and $x_0 = L_0 \cap F_0$. 

Choose some point $p \in \widetilde{M}$ which belongs to $x_0$. We assume by contradiction that the weak unstable leaf $F_0$ of $p$ is not a cylinder, which means that the backward orbit of $p$ does not accumulate in a closed orbit. In particular, if we consider $t_n \to +\infty$ we can find a sequence of deck transformations $\gamma_n \in \pi_1(M)$ going to infinity such that $p_n= \gamma_n \tilde\phi_{-t_n}(p) \to p_\infty \in \widetilde{M}$. Note that if we consider $q \in x$ and $q' \in x'$ so that they belong to the strong stable manifold of $p$ we get that if $q_n =  \gamma_n \tilde\phi_{-t_n}(q)$ and $q_n' =  \gamma_n \tilde\phi_{-t_n}(q')$ we get that the distance in $\widetilde{M}$ from $p_n$ to $q_n$ and $q_n'$ goes to infinity (because they belong to the same strong stable leaf). 

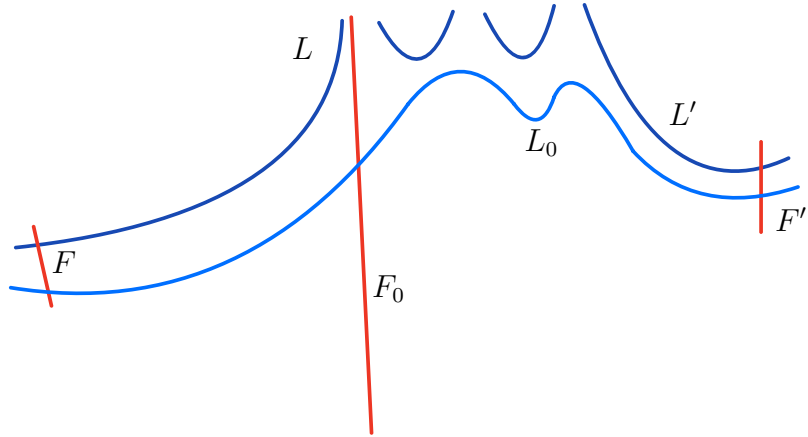
\begin{figure}[ht]
\begin{center}
\includegraphics[scale=1.15]{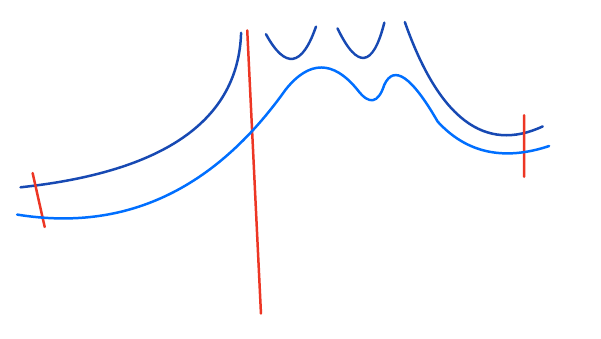}
\begin{picture}(0,0)
\put(-222,155){$L$}
\put(-85,130){ $L'$}
\put(-192,65){$F_0$}
\put(-135,120){$L_0$}
\put(-312,74){$F$}
\put(-45,90){ $F'$}
\end{picture}
\end{center}
\vspace{-0.5cm}
\caption{{\small Depiction of the elements in the proof}\label{f.nonrcov1}}
\end{figure}

The proof concludes with a figure that can now be done in the bifoliated plane, if $n,m$ are large enough, $p_n, p_m$ belong to the same foliation box. We can assume (up to relabeling) that $\gamma_n F_0$ separates $\gamma_m L$ from $\gamma_n L'$, that is, so that $\gamma_m L$ is in the connected component $U$ of $\widetilde{M} \setminus \{ \gamma_n F_0\}$ not containing $\gamma_n L'$, but this also implies that $\gamma_m F_0$ is in $U$, and therefore, one deduces that $\gamma_m L$ is contained in the connected component of $\widetilde{M} \setminus \{\gamma_n L\}$ not containing $\gamma_n L'$. Now, since the distance of $q_m'$ to $p_m$ goes to infinity, we have that $\gamma_m L'$ is contained in the same connected component of $\widetilde{M} \setminus \{\gamma_n L\}$ as $\gamma_n L'$ (because they have some closeby point) and this is a contradiction, because then $\gamma_n L$ separates\footnote{Here, we mean in the sense that they lie in different connected components of $\widetilde{M} \setminus \{\gamma_n L\}$, but this also implies that the leaves it separates cannot be 'non-separated' in the leaf space.} $\gamma_m L$ from $\gamma_m L'$ contradicting that these leaves are non-separated.

\begin{figure}[ht]
\begin{center}
\includegraphics[scale=1.15]{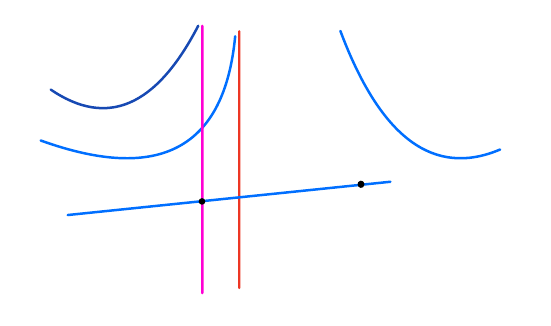}
\begin{picture}(0,0)
\put(-115,65){ $q_m'$}
\put(-205,56){$p_m$}
\end{picture}
\end{center}
\vspace{-0.5cm}
\caption{{\small A good choice of deck transformation gives a configuration that is non possible}\label{f.nonrcov2}}
\end{figure}
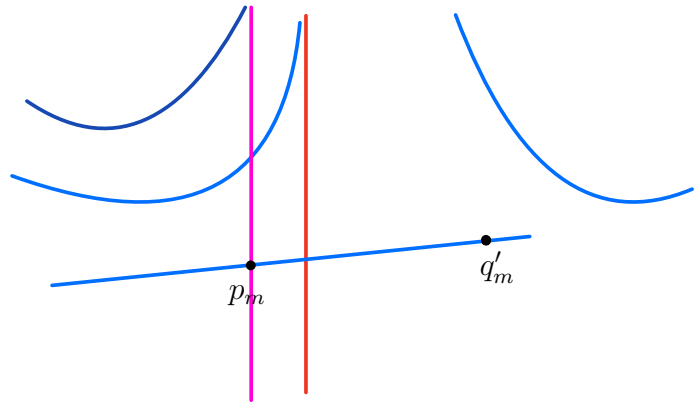

This proves that $F_0$ (and therefore $L$) is fixed by some non-trivial deck transformation. To show that the same deck transformation also fixes $L'$ requires more work, see \cite{Fenley-branch}. 
\end{proof}

The following result implies Proposition \ref{prop-periodiccornersnonsusp}. This is from \cite{Fenley-loz}. 

\begin{teo}
Let $\phi_t: M \to M$ be an Anosov flow and let $x, x'$ two different points in $\cO_\phi$ fixed by some $\gamma \in \pi_1(M) \setminus \{\mathrm{id}\}$. Then, there is a chain of lozenges joining $x$ and $x'$. 
\end{teo}

\begin{proof}[Sketch of the proof]
We will just prove that there is one lozenge, and then indicate how to conclude.  

Take $x,x'$ fixed points of $\gamma$ and let $L$ and $L'$ their respective leaves of $\widetilde{\cF^{ws}}$. 

Consider the half weak unstable leaf $A$ of $x$ which is contained in the connected component of $\widetilde{M} \setminus \{L\}$ containing $L'$. We can look at the set $\cI^s(A)$ of weak stable leaves intersecting $A$. This set cannot contain $L'$ so it is an open set, saturated by weak stables, and thus its boundary is made of weak stable leaves. There is a unique boundary leaf $L_1$ which separates $L$ from $L'$ in the boundary of $\cI^s(A)$. Note that since $\gamma$ fixes $L$ and $L'$ it also fixes $L_1$. It follows that $L_1$ contains a point $x_1$ fixed by $\gamma$. Note moreover that $L$ makes a perfect fit with $A$.

If $L_1 \neq L'$ we can continue the process to construct a sequence $L_2, \ldots, L_j, \ldots$ of stable leaves which are fixed by $\gamma$ and we claim that this process stops in finitely many steps (i.e. there is some $k$ such that $L_k = L'$). If this were not the case, we could consider the union $\cH$ of the half spaces $H_i$ of $\widetilde{M}\setminus L_i$ not containing $L'$. The set $\cH$ is an open $\widetilde{\cF^{ws}}$-saturated set, and thus the boundary has a leaf $L_\infty$ which is either $L'$ or separates $\cH$ from $L'$. Thus, it is also fixed by $\gamma$, so, if we pick a fixed point $x_\infty \in L_\infty$ we can see that its weak unstable leaf intersects infinitely many $\gamma$-fixed weak stable leaves, which is a contradiction since $\gamma$ acts as a contraction or expansion in this leaf. 

To conclude the proof, one needs to show that $x$ and $x_1$ are connected by a chain of lozenges. We will only explain briefly the strategy:

Consider the set of weak unstable leaves intersecting $L$, the set we had called $\cI^u(L)$. Note that $\cI^u(L)$ cannot contain $x_1$ because $x_1$ is fixed by $\gamma$ as is $x$ but $\cI^u(L)$ cannot contain more than one fixed point of $\gamma$. So, there is a boundary leaf $F_0$ which either contains $x_1$ or separates $L$ from $x_1$.

In the first case, we can easily see that $x$ and $x_1$ are the corners of a lozenge and we conclude the proof. 

In the second case, $F_0$, being fixed by $\gamma$ must separate $L$ from $L_1$ (i.e. they lie in distinct connected components of $\mt \setminus F_0$). The leaf $F_0$ has a fixed point $y_1$. We claim that $x$ and $y_1$ are corners of a lozenge, but this is direct since weak stable leaves intersecting $A$ must intersect a weak unstable through $L'$ and therefore intersect $F_0$, and symetrically, a weak unstable in the side of $L \setminus \{x\}$ making the perfect fit with $F_0$ must intersect the weak unstable of $y_1$. This provides the existence of one lozenge. 

One can (see \cite{Fenley-loz} for details) continue this process to produce new lozenges and as before this process will end in finitely many steps. 

\end{proof}

Finally, we have the following result:

\begin{teo}\label{teo-adjacentloz}
An Anosov flow $\phi_t$ is non-$\RR$-covered if and only if the bifoliated plane $\cO_\phi$ contains a pair of adjacent lozenges. 
\end{teo}

\begin{proof}[Sketch of the proof]
When $\phi_t$ is not $\RR$ covered we can consider $L$ and $L'$ non-separated leaves with fixed points $x,x'$ by some non trivial deck transformation $\gamma$ as given in Theorem \ref{teobranching}. 

We use the notation of the proof of the previous theorem, so, we can assume that $L_1=L'$ in that case and since $L$ and $L'$ are non-separated, we know that $x$ and $x'$ are not the corners of a lozenge. So, we need to show that the chain of lozenges claimed between $x$ and $x_1$ in the previous proof is made by adjacent lozenges and we will be done. 

Note that $y_1 \neq x'$ and thus the proof is not complete. Let us see how to construct the next lozenge. For this, we take the weak unstable manifolds intersecting the half weak stable leaf of $y_1$ (which is not the boundary of the lozenge with corners $x,y_1$) and look at the boundary leaf $F_1$ separating the set from $x_1=x'$. If $x_1 \in F_1$ then we get that there is a lozenge with corners $y_1$ and $x_1$ and which has the same half weak unstable of $y_1$ than the previous one (thus adjacent). Else, $F_1$ separates the weak unstable of $y_1$ from $x_1$ and is fixed by $\gamma$ and we get a lozenge, sharing the same side, with some new fixed point $y_2$. This gives the adjacent lozenges.  
\end{proof}

It is also proved in \cite{Fenley-branch} that when projected to $M$, only finitely many leaves of $\cF^{ws}$ can lift to branching leafs (i.e. which are non-separated from some other leaf). The results here with some additional considerations provide further structure to the branching behavior of the leaf space of the lift of an Anosov foliation to the universal cover. 

\subsection{One step up map of skewed-$\RR$-covered Anosov flows}\label{ss.onestepup}

Consider a skewed-$\RR$-covered Anosov flow $\phi_t : M \to M$ with orientability assumptions for simplicity. We use the model of the band for its bifoliated plane, that is, $\cO_\phi$ is identified with $\cB=\{(x,y) \in \RR^2 \ : \ x-1 < y < x+1\}$, the unstable foliation $\cG^u$ is identified with horizontal lines $\{y=c\}_c$ and stable foliation $\cG^s$ is identified with vertical lines $\{ x = c\}_c$. 

The orbital equivalence class of $\phi_t$ induces an action $A_\phi$ of $\pi_1(M)$ on $\cB$ which determines the flow $\phi$ up to orbit equivalence (recall Theorem \ref{teo-bifolclas}). 

We will see the following: 

\begin{prop}
There exists a homeomorphism $\eta: \cB \to \cB$ which is commutes with the action $A_\phi$ and such that preserves $\cG^s$ and $\cG^u$. This homeomorphism then induces a self orbit equivalence of $\phi_t$ to itself. 
\end{prop}

\begin{proof}
We define first $\eta_1: \cB \to \cB$ which maps $(a,b) \in \cB$ to the intersection between $\RR \times \{a+1\}$ and $\{b+1\} \times \RR$, that is, the image is $(b+1,a+1)$. Note that since $(a,b) \in \cB$ then so is this intersection. Note however that $\eta_1$ maps $\cG^s$ to $\cG^u$ and viceversa (it maps horizontals to verticals and vice versa). We will then consider $\eta = \eta_1 \circ \eta_1$ which indeed makes $(a,b) \mapsto (a+2, b+2)$ and now preservers both foliations. 

We must show that this map commutes with the action of $A_\phi$. This is just because for every $(a,b) \in \cB$ we have that $(a,b)$ and $\eta_1(a,b)$ are the corners of a lozenge, which is unique with the given orientation. Then, since $A_\phi$ is an action induced by an Anosov flow it must preserve horizontals and verticals, thus, it has to map lozenges into lozenges, and by uniqueness we get that $\eta_1 \circ A_\phi(\gamma) = A_\phi(\gamma) \circ \eta_1$. 
\end{proof}

The homeomorphism $\eta$ is sometimes called the \emph{one-step up map}. Note that in general, the map $\eta$ is not a deck transformation, but there are cases where it is (and in fact, these are characterized by being orbit equivalent to geodesic flows or finite lifts\footnote{To be more precise, if there is such a deck transformation, this implies that the deck transformation is in the center of the fundamental group, which implies that the manifold is Seifert (see for instance \cite[Appendix A]{BFFP2}). In such manifolds, Anosov flows have been clasified \cite{Ghys,Barbot-OE} and this produces the claim.})

\section{Geometry of leaves}\label{s.geomleaves} 

Here, we describe the geometry of the individual leaves, and how the orbits and the strong foliations subfoliate them. This was done for Anosov flows in \cite{Fenley} and extended for topological Anosov flows in \cite{BFP} with a different argument. Here we will sketch the ideas from \cite{BFP} in order to give a description of the flowlines inside each leaf. 

\subsection{Candel's theorem}

We recall here an important result of Candel (see \cite{CandelConlon, Calegari}). There are some subtleties on the regularity of foliations that we will ignore here and refer the reader to \cite[\S 5]{BFP} for a detailed account. 

\begin{teo}\label{teo-candel}
Let $\cF$ be a foliation without a transverse invariant measure. Then, there is a metric $g$ in $M$ which makes every leaf of constant negative curvature. In fact, to get such a metric, it is enough to have that every transverse invariant measure has negative Euler characteristic\footnote{When there is a transverse invariant measure, it makes sense to integrate continuous function to get some sort of average of the integral in leaves, where the average is taken with respect to the transverse measure. The Euler characteristic is the integral of the curvature on leaves, by a sort of Gauss-Bonnet argument. See \cite{Calegari} for more details.}. 
\end{teo}

In fact, in \cite{FP-ergo} it is shown (as an application of this result) that any minimal foliation in a manifold with non virtually solvable fundamental group must admit a metric which makes every leaf into a hyperbolic surface. For Anosov flows, we will follow a different approach (which holds also in the non-transitive case) which consists in showing by hand that every leaf is Gromov hyperbolic. 

The nice thing about having this uniformization is that it allows us to compactify each leaf of the lifted foliation to the universal cover with a circle, and this circle is useful to look at the \emph{dynamics at infinity} and sometimes obtain some coarse information that can be read inside the actual manifold.

\subsection{Gromov hyperbolicity of leaves} 

The fact that leaves are of exponential growth of volume is reminiscent of the proof of Theorem \ref{t.margulis}, the following in fact allows one to apply Theorem \ref{teo-candel}. 

\begin{lema}\label{lem-gromov}
The area of a disk of radius $R$ inside a leaf of $\wfwu$ grows uniformly exponentially with $R$. 
\end{lema}

\begin{proof}
Note that by Poincar\'e-Bendixon's theorem, a strong unstable manifold cannot intersect the same foliation box twice. Since flowing time $t$ makes the length of an unstable manifold grow of the order of $e^t$, we get that in a ball of radius $2t+1$ we have a curve of length $\sim e^t$ and since the curve cannot cross the same box twice, it means that the ball of radius $2t+2$ has area of the order of $e^t$ as desired. 
\end{proof}

This lemma allows one to deduce that every transverse invariant measure (if any) must have negative euler characteristic in order to apply Theorem  \ref{teo-candel}.  Note that in the Anosov flow case one can also use the transverse contraction/expansion of the foliations to show that there cannot be a transverse invariant measure. Transverse invariant measure cannot be atomic (because that would force compact leaves) and thus, would produce some transversal which is neither contracted nor expanded by holonomy, which is incompatible with the dynamics of an Anosov flow. 

\subsection{Flowlines are quasi-geodesic} 

In \cite{BFP} it is shown that for topological Anosov flows, flowlines are uniform quasi-geodesic inside their weak stable/unstable leaves. This extends a result from \cite{Fenley} for Anosov flows. 

\begin{prop}
Let $\phi_t: M \to M$ be a (topoloogical) Anosov flow. Then, flowlines are uniform quasi-geodesic inside their weak stable/unstable leaves. 
\end{prop}

\begin{proof}[Sketch of the proof]
We give a detailed outline of the proof because the proof given in \cite[\S 5.4]{BFP} has a problem pointed out to us by Katie Mann.  

There is a more direct way to prove this which is conceptually more satisfying. But here we just apply some results from \cite{FP-Hsdff} that give a shortcut to the result. Indeed, in that paper it is shown that if $\cG$ is a one dimensional foliation obtained as the intersection of two transverse foliations with Gromov hyperbolic leaves such that the leaf space of the lift of the foliation to the universal cover is Hausdorff must be by quasi-geodesics in each leaf. In our case, this applies thanks to Theorem \ref{teo-bifoliated} and Lemma \ref{lem-gromov}.

Note here that the results in \cite{FP-Hsdff} are stated for transverse continuous foliations tangent to continuous distributions. This is done just to simplify the setting (see in particular, \cite[Remark 5.6]{FP-Hsdff}), let us explain how to work in the general case. In general, one can always apply a smoothing to assume that one of the foliations has the desired regularity (see \cite{Calegari-smooth}). Then, one can see that the arguments in \cite{FP-Hsdff} never use the regularity of both foliations (this is pointed out in several remarks in that paper). To actually define the quasi-geodesic behavior, one needs to find a definition of quasi-geodesics in leaves that is independent on the smooth structure that one fixes, but this is quite standard, by choosing foliation boxes in $M$ and counting the number of boxes that a curve crosses in $\mt$ (note that in $\mt$ the leaves are properly embedded, so they do not intersect the same box twice if these are chosen small enough). 

\end{proof}

\subsection{Quasi-geodesic fans and the non-marker point}\label{ss.nonmarker}

Here we give a definition that describes the local picture of flowlines inside a weak stable/unstable leaf. 

\begin{defi}
Given a foliation $\cT$ by lines in a surface $H$ quasi-isometric to a hyperbolic disk $\HH^2$ we say that the foliation is a \emph{quasi-geodesic fan} if every curve in $\cT$ is a uniform quasi-geodesic and there is a point $\xi \in \partial H$ such that there is a bijection between the leaves of $\cT$ and points in $\partial H \setminus \{\xi\}$ which map each point in $\eta \in \partial H \setminus \{\xi\}$ to a leaf of $\cT$ whose endpoints are $\eta$ and $\xi$. The point $\xi \in \partial H$ will be called \emph{funnel point}  . 
\end{defi}

We also can have a description of the strong stable foliation (resp. strong unstable) inside a leaf: 

\begin{defi}
Given a foliation $\cT$ by lines in a surface $H$ quasi-isometric to a hyperbolic disk $\HH^2$ we say that the foliation is a \emph{horocyclic foliation} if the leaf space $H/_{\cT}$ is Hausdorff and there is a point $\xi \in \partial H$ so that every leaf $\ell \in \cT$ verifies that its closure in $H \cup \partial H$ is exactly $\ell \cup \{\xi\}$. 
\end{defi}

Putting together what has been discussed in the previous sections, it is not hard to obtain:

\begin{teo}\label{teo-insidestructure}
Let $\phi_t$ be an Anosov flow. Then, leaves of $\wfws$ are uniformly quasi-isometric to hyperbolic planes, the flowlines make a quasi-geodesic fan in each leaf transverse to the strong stable foliation which is an horocyclic foliation. Flowlines are oriented towards the funnel point. For leaves of $\wfwu$ the same result holds, only that flowlines are oriented against the funnel point. 
\end{teo}

\begin{figure}[ht]
\begin{center}
\includegraphics[scale=0.85]{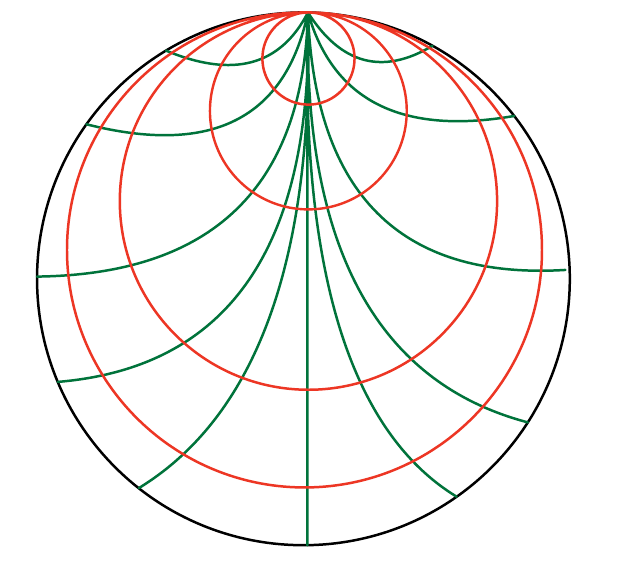}
\begin{picture}(0,0)
\end{picture}
\end{center}
\vspace{-0.5cm}
\caption{{\small The local figure of orbits and strong unstable leaves inside a weak-unstable manifold (orbits are pointing against the funnel point) }\label{f.qgfan}}
\end{figure}

It results from this theorem that for every leaf $L \in \wfws$ we can associate a point $\alpha(L) \in S^1(L)$ where $S^1(L) \cong \partial L$ denotes the Gromov boundary of the leaf $L$ (which is homeomorphic to a circle). The point $\alpha(L)$ corresponds to the funnel point to which all flowlines are directed. It has a remarkable property relating to the transverse geometry of the foliation (which is very special of Anosov foliations): 

\begin{prop}\label{prop-nonmarker}
The point $\alpha(L)$ is the unique \emph{non-marker point} of the foliation $\wfws$ in the leaf $L$. More precisely, given a quasi-geodesic ray $r \subset L$ whose endpoint in $S^1(L)$ is $\xi$ we have that: 
\begin{itemize}
\item if $\xi \neq \alpha(L)$ then the holonomy is contracting along $r$, 
\item if $\xi = \alpha(L)$ then the holonomy is expanding along $r$. 
\end{itemize}
\end{prop}

\begin{proof}
This follows by considering strong unstable arcs as transversals and noticing that along flowlines, these arcs get expanded for the future and contracted for the past. Since in $L$ every point distinct from $\alpha(L)$ is the past of a flowline we can deduce the statement. 
\end{proof}

Later we will come back to marker and non-marker points. 

\subsection{The transverse geometry of $\RR$-covered Anosov flows} 

There is a nice notion of equivalence between foliations which is coarser than being homeomorphic by a homeomorphism homotopic to the identity, but sometimes better adapted to look at large scale properties. We say that two foliations $\cF_1, \cF_2$ of a closed 3-manifold are \emph{uniformly equivalent} if for every leaf $L \in \widetilde{\cF_1}$ there is $L' \in \widetilde{\cF_2}$ such that $L$ and $L'$ are bounded Hausdorff distance apart and symmetrically, for every $F \in \widetilde{\cF_2}$ there is $F' \in \widetilde{\cF_1}$ which is bounded distance apart. Recall that $A$ and $B$ are bounded distance appart if there is $R>0$ such that the $R$-neighborhood of $A$ contains $B$ and the $R$-neighborhood of $B$ contains $A$. 

We can also define the notion of uniform $\RR$-covered foliation from \cite{Thurston}. We say that a foliation $\cF$ is \emph{uniform} $\RR$-covered if its leaf space $\cL = \mt/_{\widetilde{\cF}}$ is Hausdorff and for every pair of leaves $L, L' \in \widetilde{\cF}$ they are bounded Hausdorff distance apart. Note that if one assumes that $\cF$ is Reebless, then, it is not necessary to assume that $\cF$ is $\RR$-covered (cf. \cite{FP-min}), but otherwise it is needed \cite{Lema}. A consequence of being uniform $\RR$-covered is the following (see \cite[Theorem 9.15]{Calegari}): 

\begin{prop}\label{prop.structuremap}
Let $\cF$ be a uniform $\RR$-covered foliation of a closed 3-manifold $M$, then, there exists a map $Z: \cL \to \cL$ where $\cL = \mt/_{\widetilde{\cF}}$ which commutes with the action of $\pi_1(M)$ and has the following property: there is $c>0$ such that $d_{\mt}(L, Z(L)) > c$. 
\end{prop}

This map is sometimes called the \emph{structure} map and when the foliation $\cF$ is minimal (as is the case for Anosov foliations) it helps in showing that uniform $\RR$-covered implies that the manifold \emph{slithers} over the circle and the foliation arises as the fibers of this slithering. (We refer the reader to \cite[Chapter 9]{Calegari} for more discussion.) 

We note here that for suspension Anosov flows, the weak stable and weak unstable foliations are not uniform $\RR$-covered. This is another way to distinguish between suspensions and skewed-$\RR$-covered Anosov flows. 

\begin{prop}\label{prop-skewisuniform}
Let $\phi_t: M \to M$ be an $\RR$-covered Anosov flow. Then, we have the following properties: 
\begin{itemize}
\item $\phi_t$ is skewed-$\RR$-covered if and only if $\cF^{ws}$ is uniform $\RR$-covered. 
\item $\phi_t$ is skewed-$\RR$-covered if and only if $\cF^{ws}$ is uniformly equivalent to $\cF^{wu}$, 
\end{itemize}
\end{prop}

\begin{proof}
As mentioned, the weak stable foliation of a suspension cannot be uniform $\RR$-covered. One way to see this is to actually compute the distance between pair of leaves and see that there are points arbitrarily far apart in any pair of leaves. Alternatively, one can argue conversely, noting that every pair of leaves has points arbitrarily close (this is because the strong unstable manifold intersects every weak stable manifold in $\mt$, so flowing backwards one finds points arbitrarily close) which contradicts Proposition \ref{prop.structuremap}. Moreover, one can easily check that the one-step up map (cf. \S \ref{ss.onestepup}) works as a structure map for $\cF^{ws}$ which easily implies that the foliation has to be uniform-$\RR$-covered. 

It is also easy to see that the weak stable and weak unstable foliations of a suspension cannot be uniformly equivalent. On the other hand, if $\phi_t$ is skewed-$\RR$-covered, then any leaf of $\widetilde{\cF^{wu}}$ is contained in the region between two leaves of $\widetilde{\cF^{ws}}$, since the latter is uniform-$\RR$-covered, this shows that for every leaf of $\widetilde{\cF^{wu}}$ is bounded Hausdorff distance apart from some (and therefore all) leaf of $\widetilde{\cF^{ws}}$. The argument is symmetric, so this completes the proof. 
\end{proof}

Note that one can ask if similar results hold for non-$\RR$-covered foliations. One expects in general that the foliations will not be uniformly equivalent, but this may not be always the case (compare with Question \ref{quest-fol}): 

\begin{quest}
Are there non-$\RR$-covered Anosov flows for which $\cF^{ws}$ is uniformly equivalent to $\cF^{wu}$?
\end{quest}

I see this question as quite related to a question that was discussed quite a bit during the conference: 

\begin{quest}\label{quest-flip}
Are there non-$\RR$-covered Anosov vector fields $X$ orbit equivalent to the Anosov flow generated by $-X$? (Here of course we ask for orbit equivalence preserving orientation of the orbits.) 
\end{quest} 

This seems to be quite related to the recent results on periodic spectra, and simple periodic orbits (i.e. periodic orbits having a unique representative in their free homotopy class) from recent work \cite{BFrM,BBM}. 

\section{Universal circles}\label{s.universal}

Here we discuss two ways of producing circle actions from Anosov flows. One, corresponds to compactification of the bifoliated plane introduced by Fenley in \cite{Fenley-circle} (see also \cite{Bonatti,BBM}). The other consists in 'gluing together' the circles at infinity of individual leaves which is made in \cite{CD} (see also \cite{Thurston}). One has to do with the transverse geometry of an Anosov flow, while the other is more related to its tangencial geometry. There are some connections that we will try to explore, but also some things that are not yet completely understood to the best of my knowledge. 

\subsection{Compactifying the bi-foliated plane} 

Here we briefly explain some ideas on how to compactify the bi-foliated plane with a circle at infinity, on which the action of the fundamental group extends. This was done in \cite{Fenley-circle} and recently revisited in \cite{Bonatti,BBM}. 

We shall start by explaining the properties that such a compactification has: 

\begin{teo}[\cite{Fenley-circle,Bonatti}]\label{teo-circleFB}
Given an Anosov flow $\phi_t: M \to M$ there is a compactification of the bifoliated plane $(\cO_\phi, \cG^s, \cG^u)$ with a circle $S^1_{bif}=S^1_{bif}(\phi)$ with the following properties: 
\begin{itemize}
\item $D_\phi=\cO_\phi \cup S^1_{bif}$ is homeomorphic to $\DD$ with boundary $\partial \DD \cong S^1_{bif}$,
\item every leaf $\ell \in \cG^s$ (resp. $\ell \in \cG^u$) verifies that the closure $\overline{\ell}$ in $D_\phi$ is a closed interval whose endpoints are distinct points of $S^1_{bif}$, 
\item the action of $\pi_1(M)$ extends continuously to an action by homeomorphisms of $D_\phi$, 
\item the action in $S^1_{bif}$ is minimal if and only if $\phi$ is non-$\RR$-covered, 
\item the flow is skewed-$\RR$-covered if and only if the action has exactly two global fixed points $\xi^+, \xi^-$ and acts minimally in each connected component of the complement of $S^1_{bif} \setminus \{\xi^+, \xi^-\}$, 
\item under orientability assumptions, the flow is a suspension if and only if there are four global fixed points in $S^1_{bif}$. 
\end{itemize}
\end{teo}

\begin{proof}[Sketch of the proof] 
Let us just explain how the compactification is made, but the substantial details on why it works will be omitted. It is reminiscent of the \emph{prime end compactification} of a simply connected domain of the plane. One can define a \emph{cross cut} to be a properly embedded polygonal path made by arcs of $\cG^s$ and $\cG^u$ and notice that these separate the bifoliated plane. This way, one can make an ordering between such cross cuts by defining a sequence of decreasing cross cuts to be a sequence of polygonal paths so that the regions they define are \emph{nested} with each other and add such a sequence as a \emph{point at infinity}. Naturally, these paths and therefore these sequences are invariant under the action of $\pi_1(M)$ and so, after carefully choosing a way to identify (an equivalence relation) between sequences of polygonal paths, one gets a natural compactification, which using some properties of circle orders, can be seen to provide a circle at infinity. 

The compactification is not so hard in the product or skewed-$\RR$-covered case. In the latter, one can see the bifoliated plane as the interior of the band $B = \{ (x,y) \ : \ x-1 < y < x+1 \}$ and the foliations $\cG^s$ and $\cG^u$ to be the horizontal and vertical lines. This way, the circle at infinity is seen to be the union of the lines $y=x-1$, $y=x+1$ (on which $\pi_1(M)$ acts minimally, as they can be identified with $\cL^s$ and $\cL^u$) and two points one at $+\infty$ and one at $-\infty$ which are global fixed points of the action. The non-$\RR$-covered case is more subtle, and it is an interesting remark from \cite{Bonatti} that the minimality of the action does not detect the transitivity of the Anosov flow (but the action does determine the flow, due to \cite{BFrM,BBM}). 
\end{proof}

Recently, the concept of \emph{Anosov like actions} was proposed in \cite{BFrM}, it subsumes several of the important properties that an Anosov flow induces in the bifoliated plane, and for that reason, we include it here as a Theorem (and we refer the reader to \cite[Theorem 2.5]{BFrM} for a proof, part was somewhat explained in Proposition \ref{prop-propbifol} and \S \ref{s.nonrcov} above). We skip one of the axioms (called (A4)) because it concerns only pseudo-Anosov flows. 

\begin{teo}\label{teo.anosovlike}
Let $\phi_t: M \to M$ be a transitive Anosov flow. Then, the action of $\pi_1(M)$ on its bifoliated plane verifies the following properties: 
\begin{itemize}
\item[(A1)] If a non trivial element $\gamma \in \pi_1(M)$  fixes a leaf $\ell \in \cG^s$ or $\cG^u$, then it has a fixed point $x \in \ell$,and is topologically expanding on one leaf through $x$ and topologically contracting on the other.
\item[(A2)] The action of $\pi_1(M)$ has a dense orbit. 
\item[(A3)] The set of points $x \in \cO_\phi$ which are fixed by some $\gamma \in \pi_1(M) \setminus \{\mathrm{id}\}$ is dense. 
\item[(A5)] If two leaves $\ell_1, \ell_2$ are non-separated in the leaf space of $\cG^s$ or $\cG^u$, then, there is some non-trivial deck transformation which fixes both. 
\item[(A6)] There are no totally ideal quadrilaterals in $\cO_\phi$ (see \cite[Definition 2.11]{BFrM}). 
\end{itemize}
\end{teo}

\subsection{Marker points and universal circles} 

When $\cF$ is a taut foliation in a closed 3-manifold on which Theorem \ref{teo-candel} applies, one can associate to each leaf $L \in \widetilde{\cF}$ a circle $S^1(L)$ which corresponds to the Gromov boundary of $L$ with its Candel metric (or the metric induced by the lift of the initial metric in $M$ to $\widetilde M$, since it will be quasi-isometric to $\HH^2$ due to Candel's theorem). 

Based on some ideas from \cite{Thurston, Thurston2}, in \cite{CD} a way to \emph{glue} together all the circles $S^1(L)$ for $L \in \widetilde{\cF}$ is proposed, and this produces again a universal circle $S^1_{univ}$ on which the fundamental group of $M$ acts naturally.  

Let us explain briefly this construction restricted to the case of Anosov foliations, where this works particularly well. 

Let us define for a general foliation $\cF$ by Gromov hyperbolic leaves, what is a marker direction. The definition is not exactly the same as in \cite{CD} (see also \cite{Calegari}) but coincides for Anosov foliations (these points we will define could be called \emph{two sided markers}). 

\begin{defi}\label{defi-marker}
Given a foliation $\cF$ by Gromov hyperbolic leaves in a closed 3-manifold $M$, and a leaf $L \in \widetilde{\cF}$ we say that a point $\xi \in S^1(L)$ is a \emph{marker point} if for every $\eps>0$ there is some quasi-geodesic ray $r \subset L$ which converges to $\xi$ so that  there is an embedding $m: [0,+\infty) \times [-1,1] \to \widetilde{M}$ so that:
\begin{itemize}
\item $m(t,0)$ parametrizes $r$ by arclength, 
\item $m(\{t\} \times [1,1])$ is an arc transverse to $\widetilde{\cF}$ of length less than or equal to $\eps$, 
\item for every $s \in [1,1]$ there is a leaf $L_s \in \widetilde{\cF}$ so that for every $t \in [0,+\infty)$ we have that $m(t,s) \in L_s$. 
\end{itemize}
\end{defi} 

We can now reformulate Proposition \ref{prop-nonmarker}: 

\begin{prop}\label{prop-marker}
Let $\phi_t: M \to M$ be an Anosov flow and for $L \in \wfws$ we define $\alpha(L) \in S^1(L)$ to be the point so that every flowline in $L$ converges to $\alpha(L)$ in the future (cf. Theorem \ref{teo-insidestructure}). Then, every $\xi \in S^1(L) \setminus \alpha(L)$ is a marker point, and $\alpha(L)$ is not a marker point. Moreover, if $\xi \in  S^1(L) \setminus \alpha(L)$ we can choose the function $m: [0,+\infty) \times [-1,1] \to \widetilde{M}$ as in Definition \ref{defi-marker} so that the length of $m(\{t\}\times [-1,1])$ goes to $0$ as $t \to +\infty$. 
\end{prop}

We can consider then the tubulation at infinity, which corresponds to a way to add all the circles at infinity of leaves of $\wfws$ in $\widetilde{M}$. 

Let us first define: 

$$ S^1_\infty(\wfws) = \bigsqcup_{L \in \wfws} S^1(L). $$ 

To put a topology on $S^1_{\infty}(\wfws)$ which makes it homeomorphic to an $S^1$-bundle over $\cL^s$ and a topology on $\widetilde{M} \cup S^1_{\infty}(\wfws)$ we will just indicate the main ideas, refering the reader to \cite[\S 7.2]{Calegari} for a way to do it using the Candel metric and to \cite[\S 3]{FP-Hsdff} for a more general construction. 

The idea is that one can consider local transversals $\eta$ to $\wfws$ and give a topology to $\bigsqcup_{L \in \eta} (L \cup S^1(L))$ making it homeomorphic to $\DD \times \eta$. For this, one needs to choose a way to identify points at infinity for nearby leaves, and for this, markers are a good choice. 

Now, one can construct a universal circle by choosing a way to identify all the circles simultaneously. For this, one uses markers, that provide curves in $S^1_\infty(\wfws)$ which are transverse to the circle fibers which we want to identify. These curves, being invariant under deck transformations, produce a quotient on which $\pi_1(M)$ acts, and one can show, again using some properties about circular orders, that this quotient is a circle. This circle is called \emph{the universal circle} of $\cF^{ws}$ and denoted by $S^{1}_{univ}(\cF^{ws})$. 

The action of $\pi_1(M)$ on this universal circle is less understood. For instance, motivated by \cite{BFrM} one can ask: 

\begin{quest}
Assume that $\phi_1, \phi_2$ are two Anosov flows on $M$ with weak stable foliations $\cF^{ws}_1, \cF^{ws}_2$. Assume that the actions of $\pi_1(M)$ on $S^1_{univ}(\cF^{ws}_1)$ and $S^1_{univ}(\cF^{ws}_2)$ are conjugated, does this imply that the flows are orbit equivalent? 
\end{quest}

This is true for skew-$\RR$-covered Anosov flows as we shall see in the next section. But in general it is unclear, in particular, the actual construction of the universal circle for non-$\RR$-covered foliations is not completely cannonical (see \cite{CD}) and so one would need to make this question more precise before attempting to answer it. There are very few general results on these actions \cite{Kano} and less for Anosov foliations, so it seems that it can be a problem worth trying to understand further. In fact, already the action in the tubulation (which seems more canonical) seems interesting to check.  

Note that when the foliation is $\RR$-covered and uniform, the action is minimal \cite{FP-min} (in the next section we will explain this for Anosov foliations), which can be contrasted with Theorem \ref{teo-circleFB}. We shall explore more about the action in the next section, and try to find links between the two actions.

\subsection{$\RR$-covered Anosov flows, one step up maps and regulating flows} 
Here we discuss some results motivated by \cite{Fenley,Thurston} related to skewed-$\RR$-covered Anosov flows and their self-orbit equivalences. Some related arguments can also be found in \cite{Barbot,BG}. 

We first state the main result we will use, it was indicated by Thurston in \cite{Thurston} and given a full proof by Calegari (see \cite[Chapter 9]{Calegari}) and Fenley \cite{Fenley-regulating}. 

\begin{teo}\label{teo.regulating}
Let $\cF$ be a uniform $\RR$-covered foliation. Then, there is a pseudo-Anosov flow $\varphi_t: M \to M$ transverse to $\cF$ with the property that when lifted to $\mt$ every orbit of $\phi_t$ intersects every leaf of $\widetilde{\cF}$. 
\end{teo} 

Recall that a \emph{pseudo-Anosov flow} $\varphi_t: M \to M$ is a non-singular flow preserving two transverse singular foliations whose singularties correspond to finitely many periodic orbits of prong type (see figure \ref{f.flujopA}) and whose dynamics is contracting to the future in one of the foliations and repelling on the other. We note that it has been proved that pseudo-Anosov flows are exactly those flows which are expansive (see \cite{IM,Paternain}). From the pseudo-Anosov property, many properties of the action on the universal circle can be deduced (see \cite{FP-phpa} for discussion of some applications).

\begin{figure}[ht]
\begin{center}
\includegraphics[scale=0.85]{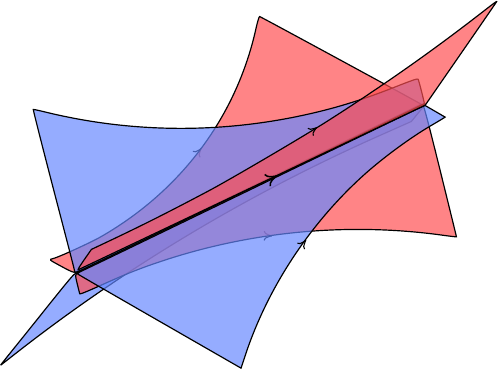}
\begin{picture}(0,0)
\end{picture}
\end{center}
\vspace{-0.5cm}
\caption{{\small Local picture a 3-prong of a pseudo-Anosov flow. (Figure by Elena Gomes.)}\label{f.flujopA}}
\end{figure}

\begin{proof}[Sketch of the proof]
We actually recommend looking at \cite{Thurston} to get a good idea on how this problem is approached. In fact, as it is very well explained in \cite{Calegari}, the proof mimics the ideas in the proof by Thurston that homeomorphisms of surfaces not having any free homotopy class of simple closed curves periodic, must have a pseudo-Anosov representative. Here, these periodic free homotopy classes are replaced by incompressible tori in $M$ (which do not exist if $M$ is hyperbolic). 

The role of iterating with lifts of the surface diffeomorphism is changed by the use of the structure map $Z$ which has a well defined realization as a map that extends to the universal circle acting as the identity, and since $Z$ commutes with deck transformations, one can look at the action of $\pi_1(M)$ to produce two transverse invariant laminations whose complementary regions are ideal polygons. The strategy is very similar to the surface homeomorphism case, one uses Candel's metric on leaves, and if one takes a geodesic there, the image by $Z$ is a quasi-geodesic that one can stretch to a geodesic, and the action of $\pi_1(M)$ maps geodesics to geodesics, so, the joint action of $\pi_1(M)$ and $Z$ on the space of 'laminations' of the universal circle is well defined. One needs to show that taking limits one obtains an invariant lamination (i.e. it does not depend on the subsequences) and that if there is a complementary region which is not an ideal polygon, then one can construct an incompressible torus. Finally, one needs to make some hyperbolic geometry arguments by looking at moduli of quadrilaterals to obtain the more precise expansion/contraction properties to finish the proof. 
\end{proof}

We note that recently Fenley \cite{Fenley-regulating} extended this to all 3-manifolds admitting a uniform $\RR$-covered foliation obtaining results very much analogous as those for the Nielsen-Thurston classification of diffeomorphisms of surfaces.

For an $\RR$-covered Anosov flow $\phi_t: M \to M$, we want to relate the action of $\pi_1(M)$ in $S^1_{univ}(\cF^{ws})$ with the action of $\pi_1(M)$ on $S^1_{bif}(\phi)$. Recall that the action of $\pi_1(M)$ on $S^1_{bif}(\phi)$ is with two global fixed points, and the action in the complement is given by the action of $\pi_1(M)$ on the leaf space $\cL^s$ (which is conjugated to the action on $\cL^u$). 

So, the following result identifies these actions: 

\begin{prop}\label{prop-actionskewR}
Let $\phi_t: M \to M$ be a skewed-$\RR$-covered Anosov flow. Then, the $\pi_1(M)$ action on $S^1_{univ}(\cF^{ws})$ is conjugated to the $\pi_1(M)$ action on the circle $\cL^s/_{Z^s}$ (where $Z^s$ denotes the structure\footnote{The way we defined the structure map, it is not uniquely defined, however, there is a cannonical one, which, in the case of skewed-$\RR$-covered Anosov flows is in fact the one given by the one step up map. This is the one we use.} map of the foliation $\widetilde{\cF^{ws}}$).  
\end{prop}

\begin{proof}[Sketch of the proof]
The non-marker point defined in \S~\ref{ss.nonmarker} allows one to define a map $\alpha: \cL^s \to S^1_{univ}(\cF^{ws})$ mapping each leaf into a point in the universal circle corresponding to the non-marker point of the leaf. This map is a covering map, $\pi_1(M)$ equivariant and commutes with the one-step up map, thus, it defines the desired identification. 
\end{proof}

More discussions on these properties can be found in \cite{Thurston,BFFP,FP-min,Fenley-regulating,FP-phpa}.  In particular, when $M$ is hyperbolic, the fact that the regulating flow is pseudo-Anosov can be used to obtain properties of this action, for example, one can show: 

\begin{prop}
Let $\phi_t: M \to M$ be a skewed-$\RR$-covered Anosov flow in a hyperbolic 3-manifold. Then, there exists $k>0$ such that for every $\gamma \in \pi_1(M)$ the action of $\gamma^k$ in $S^1_{univ}(\cF^{ws})$ has finitely many fixed points, alternatively attracting and repelling. 
\end{prop}

\begin{proof}[Sketch of the proof]
It is good to remark first that unless $\gamma$ corresponds to a fixed leaf (i.e. represents a free homotopy class of a periodic point) in which case $\gamma$ acts with one attracting and one repelling fixed point, the element $\gamma$ acts as a translation in $\cL^s$. Comparing with the pseudo-Anosov flow is the the way to adress these other deck transformations. 

The value of $k$ corresponds to a common multiple of the number of prongs of the singularities of the transverse regulating pseudo-Anosov flow. This way, deck transformations associated to periodic orbits will fix the points at infinity corresponding to the limits of the liftts of the rays obtained by intersecting the singular foliation with the leaves of $\widetilde{\cF^{ws}}$. For deck transformations associated to elements which are non-periodic, one can show that there are exactly two fixed points, attracting and repelling respectively (see e.g. \cite[Lemma 8.5]{BFFP}). 
\end{proof}

More discussion on the action of regulating flows, even in more general manifolds, can be found in \cite[\S 13]{FP-phpa}.  Note that there, the notion of super attracting and super repelling fixed points is introduced. This allows to give a proof that if $\phi_t: M \to M$ is a skewed-$\RR$-covered Anosov flow in a hyperbolic 3-manifold and $\psi_t: M \to M$ is the regulating pseudo-Anosov flow, then, there is no $\gamma \in \pi_1(M)$ so that it represents both periodic orbits of $\phi_t$ and $\psi_t$. One can however ask: 

\begin{quest}
Let $\phi_t^1, \ldots, \phi_t^k: M \to M$ be Anosov flows (one can restrict to skewed-$\RR$-covered Anosov flows and this is already interesting), is it true that there exists $\gamma \in \pi_1(M)$ which represents periodic orbits of \emph{all} the flows? 
\end{quest}

This could have applications in the study of transverse foliations \cite{FP-t1s,FP-Hsdff,BarbotFenleyPotrie} and is certainly very related with the problems introduced in \cite{BM,BFrM}.

\section{Examples}\label{s.examples}

There are several references which explain the construction of examples of Anosov flows. The classical ones are covered very nicely in \cite{KH, FH} and are also reviewed in \cite{Barthelme}. Surgery constructions are reviewed in \cite{Barthelme,FH} besides the original references \cite{HandelThurston,Fried,Goodman, FoH,Sa}. Constructions with building blocks are also reviewed in \cite{Barthelme,FH} besides the original references \cite{FW,BL,BBY,Paulet} and we recommend \cite{Paulet} for a general presentation making interesting interaction between constructions by surgery and building blocks. Constructions of examples by these techniques keeping control on the underlying manifold, or the resulting leaf spaces has been done in \cite{Fenley, BoM, ClayPinsky, BI, BY} and others.  Here, we will not give many details on the way examples are constructed or how they are proved  to be Anosov (which is done carefully in the mentioned references) but rather emphasize on some of the features the examples we will present have. This will allow us to have concrete instances of the results that have been surveyed above. 

One interesting exercise (which the author of this note has not done completely) is to try to understand the explicit characteristics of each example in terms of their leaf spaces, the possible actions both on those spaces and in their universal circles, etc. 

\subsection{Classical examples}

\subsubsection{Suspension flows}

Consider $A \in \mathrm{SL}_2(\ZZ)$ a hyperbolic matrix. It induces a diffeomorphism $f_A: \TT^2 \to \TT^2$ by considering $\TT^2 \cong \RR^2/_{\ZZ^2}$ and $f_A(v + \ZZ^2) = Av + \ZZ^2$ for every $v \in \RR^2$.  

We can construct $M_A = \TT^2 \times \RR/_{\sim}$ where $\sim$ is the equivalence relation generated by $(x,s+1) \sim (f_A(x), s)$ for every $(x,s) \in \TT^2 \times \RR$ and consider the flow $\varphi_t(x,s) = (x, s+1)$ which respects the equivalence relation and therefore induces a flow in $M_A$. 

It is a standard argument (see e.g. \cite{Barthelme}) to show that the flow $\varphi_t : M_A \to M_A$ is Anosov. The foliation $\cF^{ss}$ is given by the projection to $M_A$ of the lines of the form $(v  + E^s_A + \ZZ^2 , s)$ where $E^s_A$ is the eigenspace associated to the eigenvalue of modulus smaller than $1$ of $A$. 

A relevant feature of this example is that we can understand quite well the geometry and topology of the invariant foliations. When intersected with each torus $\TT^2 \times \{s\}$ the weak stable foliation is an irrational foliation. It follows that there is a \emph{global product structure} between the weak stable foliation and the strong unstable foliation; this means that for every pair of points $x,y \in \widetilde{M_A}$ (the universal cover of $M_A$) we have that $\widetilde{\cF^{ws}}(x) \cap \widetilde{\cF^{uu}}(y) \neq \emptyset$ and consists of exactly one point. 

This example is the $\RR$-covered case which is not skewed (recall Theorem \ref{teo.rcoveredboth}), here it is easy to understand the bifoliated plane, as any connected component of the lift of the transverse torus to the universal cover, intersects every orbit of the lifted flow. The weak stable and weak unstable foliations intersect the torus in two transverse linear foliations.

\subsubsection{Geodesic flows and their finite lifts} 

One can consider $\mathrm{PSL}_2(\RR)$ identified with the isometry group of $\HH^2$. If $\Gamma < \mathrm{PSL}_2(\RR)$ is a discrete and cocompact subgroup, we know that $M = \mathrm{PSL}_2(\RR)/_{\Gamma}$ is a closed 3-manifold. Existence of such subgroups can be justified by geometric considerations (for instance, the uniformization theorem for higher genus surfaces) or arithmetic ones. 

One can easily show that the action on the right of the one-parameter group $a_t = \begin{pmatrix} e^{t/2} & 0 \\ 0 & e^{-t/2} \end{pmatrix}$ is an Anosov flow on $M$. 

Here $\mathrm{PSL}_2(\RR)$ identifies with the intermediate cover $T^1\HH^2$. One can identify $T^1\HH^2$ with $\HH^2 \times \partial \HH^2$ via $(x,v) \mapsto (x, v_+)$ where $v_+ \in \partial \HH^2$ is the point at infinity which is the limit of the geodesic with initial conditions $(x,v)$. Note that with this identification, one has that the weak stable foliation lifts to $\HH^2 \times \partial \HH^2$ as the foliation $\widehat{ \cF^{ws}} = \{ \HH^2 \times \{\xi\}\}_{\xi \in \partial \HH^2}$. This allows easily to show that the foliation is $\RR$-covered. Also, since the strong unstable manifold of a point in $p \in \HH^2 \times \{\xi\}$ can be seen to intersect every leaf in a unique point (in particular, intersect $ \HH^2 \times \{\xi\}$ only at $p$) we get that in the universal cover, there is a skewed picture for the foliations. The perfect fit (and lozenges) made by different lifts of the same leaf. 

Another way to understand this is to parametrize $T^1 \HH^2$ as a pair of distinct points in $\partial \HH^2$ and a real number (called Hopf parametrization). With this parametrization, one simultaneously identifies the weak stable and weak unstable foliations as horizontal and vertical segments in $(\partial \HH^2 \times \partial \HH^2) \setminus \{(x,x) :  x \in \partial \HH^2\}$ which lifts to the bifoliated plane. 

Note that one can also consider lifts of these representations to finite covers $\mathrm{PSL}_2^{(k)}(\RR)$ which exist in many cases and provide finite covers of the geodesic flow. These examples have some relevance in the study of general classification of Anosov flows on Seifert manifolds \cite{BarbotFenley3}, and also in  construction of hyperbolic manifolds admitting many distinct Anosov flows \cite{BoM}.

\subsection{Transverse tori}

The first examples of non-suspension Anosov flows which have a torus transverse to the flow (recall Proposition \ref{prop-transversetorus}) were produced by Franks and Williams \cite{FW} giving the first example of non-transitive Anosov systems. The non-transitivity was indeed quite useful to ensure the fact that the flow was Anosov for reasons that we shall explain. After, that, in \cite{BL} a transitive example was produced and more recently, a very general technique was devised in \cite{BBY} which allows to glue several pieces and construct a myriad of different examples (all of which admit transverse tori and are not suspensions). Recently, in \cite{Paulet} this technique (as well as others) were further pushed to the strongest current version, somewhat subsuming both the constructions from \cite{BBY, BL} on the one hand and \cite{HandelThurston} on the other. We will make a very short cartoon of this, and refer the reader to \cite{Paulet} for details and more constructions. (The constructions are also surveyed in \cite{Barthelme,FH}.) 

In sum, one has the following data: 

\begin{itemize}
\item a \emph{block} is a compact 3-manifold $B$ with boundary $\partial B = \partial_{out} B \cup \partial_{in} B$ which is a finite union of tori and contains a vector field $X$ which is pointing outwards in every component of $\partial_{out} B$, pointing in at every component of $\partial_{in}B$ and so that the maximal invariant set $\Lambda$ in $B$ (that is, the set $\Lambda = \{x \in B \ : \ X_t(x) \in B \setminus \partial B \ \forall t \in \RR \}$) is a hyperbolic set. 

\item a pattern in the boundary is, on each component $T$ of $\partial_{out} B$ the lamination $\Gamma^u_T$ obtained as the intersection of the weak unstable lamination $\cF^{wu}(\Lambda)$ of $\Lambda$ with $T$, and on each component $T'$ of $\partial_{in}B$ the lamination $\Gamma^s_{T'}$ obtained by intersecting the weak stable lamination $\cF^{ws}(\Lambda)$ of $\Lambda$ with $T'$. 
\end{itemize}

A \emph{piece} is then a 6-tuple $(B,X, \partial_{out}B, \partial_{in} B, \Gamma^s, \Gamma^u)$ where $(B, X, \partial_{out}B, \partial_{in}B)$ is a block and $\Gamma^s$ and $\Gamma^u$ are a pattern in the boundary components. 

It is shown (\cite{FW,BL,BBY,Paulet}) in different types of generalities (with some conditions on the laminations and the patterns) that if one has a family of pieces, and a way of gluing the outward boundaries with inward boundaries that this produces an Anosov flow. 

Note that as explained above Proposition \ref{prop-transversetorus} implies that these examples cannot be $\RR$-covered. It is relevant to try to understand how the torus is represented in the orbit space $\cO_\phi$ of the resulting flow. Note that since it is transverse to the flow, when lifted to the universal cover the transverse torus gives several planes which are transverse to the flow. However, there are several orbits that intersect these planes many times, so, some deck translates of the plane intersect the plane, but each representation of the plane has a $\ZZ^2$ which stabilize it. It is usually the case that many lozenges can be found in these tori. 

The tori allow one to describe part of the orbit space and the bifoliated plane, but as far as I know, we still have space to understand this plane a bit better.

\subsection{Surgeries} 

Here we briefly describe the surgery construction due to Goodman \cite{Goodman}, which admits a dfferent interpretation due to Fried \cite{Fried} (but provides the same examples, thanks to \cite{Shannon}). 

The idea is to start with an Anosov flow, pick some periodic orbit $o$, blow it up, and then blow it down with some operation that will change the topology of the manifold, but will keep the same flow in $M \setminus o$. This operation requires a periodic orbit $o$, and some integer $m \in \ZZ$ which determines the way the blow down is made. In \cite{FoH,Sa} it is shown that in some cases, the operation can be done by keeping the fact that the flow is the Reeb flow for some contact form. It was the way to construct the first \cite{Goodman}, and all the currently known (see for instance \cite{Fenley,BoM,BI,Fenley-QG}), examples of Anosov flows in hyperbolic 3-manifolds. 

In \cite{Fenley,BI,ABM} interesting properties of the effect of surgeries in the leaf spaces are considered. In particular, related to the problem of understanding the \emph{graph} of Dehn-Goodman-Fried surgeries \cite{De,DeSh} it is shown that every Anosov flow is connected via some surgery to some $\RR$-covered one (\cite{ABM}).

\section{Contact Anosov flows and leaf spaces}\label{s.contact}

Here we will explain a beautiful result due to Barbot \cite{Barbot-tcp} saying that contact Anosov flows must be $\RR$-covered. This has been recently complemented in \cite{Marty} where the converse has been proved: if an Anosov flow is $\RR$-covered, then, it is orbit equivalent to a contact Anosov flow\footnote{Note that being contact is sensible to reparametrizations, so, one cannot expect more than this.}. 

We note here that the result of Marty \cite{Marty}, stating that every skewed-$\RR$-covered Anosov flow is orbit equivalent to a contact Anosov flow, complemented with some recent result of Barthelm\'e, Bowden and Mann \cite{BM}, implies finiteness of $\RR$-covered Anosov flows in a given 3-manifold. This answers an important case of a quite relevant and still open problem: 

\begin{quest}
Is there a $3$-manifold admitting infinitely many non orbit equivalent Anosov flows?
\end{quest}

We state then this important result. 

\begin{teo}[Barbot-Marty]\label{teo-BarbotMarty}
Let $\phi_t: M \to M$ be a an Anosov flow in a closed 3-manifold. Then, it is skewed-$\RR$-covered if and only if it is orbit equivalent to a contact Anosov flow. 
\end{teo} 

We briefly explain the idea of Barbot's proof of the converse implication \cite{Barbot-tcp}. The proof divides in two steps. The first part is to use the contact property to deduce that the strong stable leaf space in $\widetilde{M}$ is Hausdorff (and therefore a plane). Then, using this and the existence of adjacent lozenges (which characterize non-$\RR$-covered Anosov flows, Theorem \ref{teo-adjacentloz}) deduce that a non-$\RR$-covered Anosov flow cannot be contact. Since the orbit space of the Anosov flow can be $\RR^2$ even when the weak foliations leaf spaces are non-Hausdorff, one can naively ask: 

\begin{quest}
Does every Anosov flow in dimension 3 have the property that the strong stable leaf space is homeomorphic to $\RR^2$? 
\end{quest}

As pointed out to me by Katie Mann, Chi Cheuk Tsang and Bojun Zhao after these notes where available, it is very unlikely that the question above has a positive answer (in fact, they have a convincing argument that the leaf space of the strong stable is never $\RR^2$ unless the Anosov flow is $\RR$-covered, using similar arguments as in \cite{Barbot-tcp}). In fact, they suggested (and it seems reasonable) that the leaf space could be homeomorphic to the weak stable leaf space times $\RR$. This seems to interact interestingly with Question \ref{quest3}. Note also that flows obtained by parametrizing the strong stable foliation of an Anosov flow are one of the few sources of minimal flows in 3-manifolds (see the introduction of \cite{HandelThurston}), which is the reason I believe understanding the leaf space of such foliations is relevant if one wishes to better understand the possible minimal flows in dimension 3. 

In \cite{Marty} a direct, and more conceptual proof of the converse implication in Theorem \ref{teo-BarbotMarty} was obtained, but the proof is hidden inside that paper. In the note \cite{Marty2} that will appear in this proceedings, the proof is explained in detail.

\section{Other directions}\label{s.otherthings}

This text is very introductory and only mentions briefly some recent advances. Moreover, it is quite focused in the topology and geometry of Anosov flows. We have thus neglected several lines of research related to Anosov flows and failed to mention several recent results even in the chosen focus. I will mention just a few, based purely in personal taste. 

In particular, we barely mentioned many of the most recent related to the further study of the orbit spaces and the bifoliated plane, which build on the basic results we have discussed. See \cite{BFrM, BFeM, BBM} for introductions to these beautiful developments (we point out here that these developments also include the study of pseudo-Anosov flows which are fascinating object that we touched upon only very briefly here).  There is some relation with a deep theory developed by Barbot and Fenley (see \cite{BarbotFenley,BarbotFenley2} and references therein) which produces somehow a JSJ decomposition for Anosov flows in a series of works which are described also in \cite{Barthelme} (together with some consequences explored in particular by Barthelme and Fenley). I find the study, initiated in \cite{BG}, of what is called \emph{self orbit equivalences} of Anosov flows quite fascinating, and has been connected closely to the theory of partially hyperbolic diffeomorphisms (see \cite{HP,BFP}).

We briefly mentioned at the end the problem of understanding minimal flows on 3-manifolds and its relation with Anosov flows that was first (as far as I am aware) pointed out in \cite{HandelThurston}. The general problem of understanding flows and homeomorphisms of 3-manifolds is wide open but we predict that understanding Anosov flows can provide insight in the problem both as examples and as models to understand their behavior. In \cite{FP-t1s, FP-Hsdff, BarbotFenleyPotrie} we started a systematic study of the flows that arise as the intersection of two transverse foliations and we have spotted several links to Anosov flows. In this conference, we have attended to some exiting developments on the relation between Anosov flows and bi-contact structures, triggered by work of Mitsumatsu \cite{Mitsumatsu} and more recently Hoozori \cite{Hoozori}. This has motivated some new exiting results that should be payed attention to (see for instance \cite{CLMM,Massoni} and references therein).  The study of bi-contact structures and flows arising in the intersection of transverse foliations is certainly connected, but the connection is still not very clearly understood.

In hyperbolic 3-manifolds, besides the construction of manifolds admitting many distinct Anosov flows in \cite{BoM,BY}, it makes sense to ask about the properties of the extension of orbits to the sphere at infinity and its relation with the circle at infinity of each leaf of the weak foliations. Recently, Fenley \cite{Fenley-QG} has shown that \emph{every} non-$\RR$-covered Anosov flow in a hyperbolic 3-manifold is quasi-geodesic (completing the characterization of quasi-geodesic flows thanks to \cite{Fenley}) which is surprising non only for its beauty but because no examples of Anosov quasi-geodesic flow in a hyperbolic 3-manifold was known and because it exposes a deep relation between the leaf space of an Anosov flow and its geometric properties. The study of pseudo-Anosov flows in hyperbolic 3-manifolds and its relation to taut foliations is also very active and hard to describe with a few references, but let us point to \cite{LMT} for a recent nice result with several references to the literature. 

There is also an influential conjecture due to Ghys relating Dehn-Goodman-Fried surgery that has not been touched upon. We refer the reader to \cite{De,DeSh} for an introduction to that problem and to the theory of Birkhoff sections which are also quite related to the concept of \emph{open book decompositions}.

But also, I want to point out other work that relates to Anosov flows from more dynamical or analytic viewpoints, for this, I recommend the recent work on exponential mixing \cite{TsujiiZhang} and the recent advances on microlocal techniques to understand Anosov flows and which have deep connections with topology (related to helicity and linking number) which would be nice to understand further, we refer the reader to \cite{CePat} for a recent introduction pointing also to other relevant references.


\end{document}